\documentclass[12pt, reqno]{amsart}
\usepackage[utf8]{inputenc}
\usepackage{amsthm, amsmath, amssymb, tikz, bm,diffcoeff}
\usepackage{verbatim}
\usepackage{graphicx}
\usepackage[linesnumbered,ruled,vlined]{algorithm2e}
\usepackage{float}
\usepackage{derivative}
\usepackage{mathtools}
\usepackage{CJK} 
\usepackage{xifthen}
\usepackage[colorlinks=true,
linkcolor=blue,citecolor=blue,
urlcolor=blue]{hyperref}

\usepackage[margin=1.2in]{geometry}
\usepackage{extarrows}

\usepackage[shortlabels]{enumitem}
\usepackage{todonotes}
\usetikzlibrary{calc,shapes, backgrounds}
\allowdisplaybreaks

\newtheorem{lemma}{Lemma}
\newtheorem{theorem}{Theorem}
\newtheorem{corollary}{Corollary}
\newtheorem{definition}{Definition}

\newtheorem{remark}{Remark}
\newtheorem{conjecture}{Conjecture}
\newtheorem{proposition}{Proposition}
\newtheorem{example}{Example}
\usepackage{xcolor}

\title{On the Ricci flow on Trees}

\author{
Shuliang Bai}
\thanks{Beijing Yanqi Lake Institute of Mathematical Sciences and Applications, China. 
Email address: baishuliang@bimsa.cn, corresponding author. }
 \author{Bobo Hua}
\thanks{School of Mathematical Sciences, LMNS, Fudan University, Shanghai
200433, China; 
Email address: bobohua@fudan.edu.cn.}

\author{Yong Lin} 
\thanks{Yau Mathematical Science Center, Tsinghua University, Beijing, 100084, China; Department of Mathematics,  Tsinghua University, Beijing, 100084, China.
Email address: yonglin@tsinghua.edu.cn.}
\author{Shuang Liu}
\thanks{School of Mathematics, Renmin University of China.
Email address: shuangliu@ruc.edu.cn.}

\date{}

\begin{document}
\maketitle

\begin{abstract}
In this paper, we study the evolution of metrics on finite trees under continuous-time  Ricci flows based on the Lin-Lu-Yau version of Ollivier Ricci curvature. We analyze long-time dynamics of edge weights and curvatures, providing precise characterizations of their limiting behaviors.  We prove that the Ricci flow converges to metric with zero curvature on edges whose normalized weights converge to positive values only if the tree is a caterpillar tree.  
\end{abstract}

\section{Introduction}




In Riemannian geometry, the Ricci flow initiated by Hamilton \cite{hami1982} is a powerful tool for studying the global structure of manifolds, which has numerical applications in general relativity and geometric analysis. The Ricci flow is an evolution equation for metrics on a Riemannian manifold, defined by:
\[ \frac{\partial g_{ij}}{\partial t} = -2 \text{Ric}_{ij} \]
where $ g_{ij} $ is the metric tensor and $ \text{Ric}_{ij} $ is the Ricci curvature tensor. The normalized Ricci flow, which adjusts the volume of the manifold to remain constant, is given by the following:

\[ \frac{\partial g_{ij}}{\partial t} = -2 \text{Ric}_{ij} + \frac{2}{n} R g_{ij} \]
where $ R $ is the scalar curvature and $ n $ is the dimension of the manifold. Hamilton \cite{hami1982} proved that, for a closed 3-manifold, starting from an initial metric with positive Ricci curvature, the normalized Ricci flow evolves into a round metric, emphasizing the profound relationship between geometric evolution and curvature. The fundamental contributions of the Ricci flow are the solutions of the Poincar\'e conjecture,  Thurston’s geometrization conjecture and the differentiable sphere conjecture, see \cite{GP1,GP3,GP2,HRF06,CZ06,MT07,KL08,BS09,MT14,RFTA1,RFTA2,RFTA3,RFTA4}, etc.

In graph theory, the notion of discrete curvature has become an important concept for understanding geometric properties of graphs, particularly when these graphs are equipped with weights on the edges. The development of curvature notions, such as the Bakry-Émery curvature \cite{YongLinYau}, the Ollivier Ricci curvature \cite{Ollivier}, the Forman Ricci curvature \cite{F}, etc., has enabled researchers to analyze and interpret structural properties of weighted graphs in ways that extend classical ideas from differential geometry to the discrete setting. In this paper, we study the continuous-time Ricci flow on trees based on the Lin-Lu-Yau version of Ollivier Ricci curvature \cite{LLY}.

The (un-normalized) Ricci flow on a graph $G=(V(G),E(G))$ introduced in \cite{bailin} is given by the following system of equations 
\begin{align}\label{eq:unnor_continuous}
 \frac{\partial {w}_e(t)}{\partial t} = -\kappa_e(t) {w}_e(t),   \end{align}
where ${w}_e(t) $ represents the weight of edge $e$ at time $t$ and $ \kappa_e(t) $ represents  the Lin-Lu-Yau version of Ollivier Ricci curvature on edge $ e $. The normalized Ricci flow on graphs, which adjusts the total edge weights to remain constant $1$, is described by the following system of equations
\begin{align}\label{eq:stochastic_continuous}
\frac{\partial w_e(t)}{\partial t} = -\kappa_e(t) w_e(t) + w_e \sum\limits_{h \in E(G)} \kappa_h w_h(t),
 \end{align}
where $ w_e(t) $ represents the normalized weight of edge $ e $ at time $ t $. 
For the above Ricci flows, it has been shown in \cite{bailin} that the evolution is well-defined and the solution exists for all time. This system of equations captures the dynamic evolution of the metric (edge weights) on a graph over time. Curvatures $\kappa_e(t)$ influence the rate of change of edge weights $w_e(t)$, with negative curvature leading to a decrease in weight and positive curvature leading to an increase. This behavior aligns with the intuitive understanding of the Ricci flow, where negative curvature tends to ``expand'' the space while positive curvature tends to ``shrink'' it.
By studying the evolution of the metric over time, one can gain insight into how the geometry of the graph changes, including the evolution of distances between vertices and other geometric properties.

Understanding the limits and behavior of Ricci flow solutions on graphs opens the door to a wide range of applications across multiple disciplines. In complex network modeling, the Ricci flow provides insights into the underlying structure and dynamics of networks, aiding in tasks such as community detection \cite{NLLG}, network alignment \cite{NLGG18}, data mining \cite{10.1093/comnet/cnw030} and so on.

For the theoretic part, the Ricci flow offers a novel perspective on geometric properties in discrete spaces, bridging the gap between differential geometry and graph theory. 
In the study of discrete curvature and Ricci flow, forman-Ricci flow has been applied to change detection in complex networks \cite{axioms5040026}, where the authors develop a geometric method to characterize dynamic effects in networks based on Forman's discretization of Ricci curvature. 
Cushing et al. in \cite{Liu2023BakryEmery} systematically explored the behavior of Ricci flow on graphs based on the Bakry-Émery curvature definition. They explore the interplay between graph structure and curvature evolution, providing valuable insights into how curvature conditions influence graph metrics and dynamics. See \cite{HLW24} for a variant of Bakry-Émery Ricci flow. For the discrete-time Ollivier Ricci curvature flow on a finite weighted graph, Li and Münch in \cite{RuoweiLiFlorentin} proved that the flow under the surgery converges to a constant curvature metric via a convergence result of a general nonlinear Markov chain with the monotonicity property.  Other flows on graphs and discrete settings have been developed to explore various aspects of discrete geometry and dynamics; see \cite{DG2003, flowonsurface, GLICKENSTEIN2005791, 10.1093/comnet/cnw030,  CombinatorialCalabiFlows, gu2018discreteII,  superflow, ge2021circle, comflow, MaYang2025ModifiedRicci,MaYang2025PiecewiseRicci}.

The key issue at hand is the Ricci flow's ability to converge. If the solutions converge, this suggests that a metric featuring constant curvature can be found.
To prove the convergence, classical techniques, such as gradient flows and Lyapunov stability analysis, are often useful. However, since there is no variational structure found for the Ricci flow on graphs, their direct applications seem to not work. 
In contrast to previous works, we restrict ourselves to the class of trees and prove the convergence of the Ricci flow via the combinatorial advantage of trees.

Trees are fundamental types of graphs, characterized by the absence of cycles. This simplicity allows the curvature to evolve in a predictable way, which makes trees particularly suited for the application of Ricci flow dynamics. 
The study of Ricci flow on trees is an essential first step in tackling the Ricci flow on general graphs. Understanding the Ricci flow on trees is a crucial foundation for further applications in discrete geometry and network analysis.

\subsection{ Main results.}

We first explore key features of the unnormalized weight and curvature, which are fundamental attributes valuable for further research.     We say that a solution to the Ricci flow {\it converges} 
     if  $\lim\limits_{t\to \infty}w_h(t)$ exists, being possibly infinite, for every edge $h\in E$. 
     And the limit of $w(\infty)$ is called the limit metric. 
     
     In our study, the curvature of a weighted graph is defined using a function $\gamma$, which plays a key role in assigning probabilities to the edges connected to each vertex.  In this paper, we usually select the parameter function $\gamma(x) = \frac{1}{x}$, see Section~\ref{sec:pre} for details. 
\begin{theorem}\label{thm:foundproperties}
For a tree  $T=(V, E)$ with an initial metric, the Ricci flow  \eqref{eq:unnor_continuous} converges. Moreover,  if $\lim\limits_{t\to\infty} w_e(t)= \infty$, then  for sufficiently large time $t\gg 1,$ $w_e(t)$ is increasing in $t$. 
\end{theorem}
The demonstration proof of this theorem will be based on several propositions outlined below. 

\textbf{Key intermediate results on the weights are as follows:}
\begin{enumerate}  
        \item For a leaf $e$ and an internal edge $f$ incident to the same vertex, the internal edge's weight remains larger than the leaf's weight for large time, that is, 
        $$w_e(t)<w_f(t),\quad t\gg 1. $$
    
     \item The weights of two leaves $e, g$ incident to the same vertex 
evolve such that their initial order is preserved over time, their difference tends to zero and their ratio converges to $1$ as $t$ goes to infinity, that is, 
$$(w_e-w_g)(t)=(w_e(0)-w_g(0))e^{-t},$$
and 
$$\lim\limits_{t\to \infty} \frac{w_e(t)}{w_g(t)}= 1. $$

    \item For each edge $h$, if it holds that $\lim\limits_{t\to \infty}w_h(t)=\infty$, then $w_h(t)$ is increasing in $t$ for $t\gg1.$  
    
\end{enumerate}    
    


\textbf{Key intermediate results on the curvatures are follows:}

If $h$ is a leaf edge incident to internal vertex $u$,  there are four cases:
\begin{enumerate}  
\item $\kappa_h(t)\equiv 1$ if and only if  the degree $d_u=2$; 
\item $\lim\limits_{t\to \infty} \kappa_h(t)=\frac{1}{d_u-1}$ if and only if $u$ is incident to $d_u-1$ leaves;
\item $\lim\limits_{t\to \infty} \kappa_h(t)=0$ if and only if $u$ is incident to $d_u-2$ leaves; 
\item $\kappa_h(t)\leq -\frac{d'_u}{d_u}$ for large time $t$ if and only if $d_u>3$ and $u$ is incident to at most $d_u-3$ leaves.
\end{enumerate}  

If $h=uv$ is an internal edge where either $d_u=2$ or it connects to at least $d_u-2$ leaves (similarly for $v$), then $\lim\limits_{t\to \infty} \kappa_h(t)=0$.
    
In a weighted graph, we say the normalized Ricci flow converges to a metric with {\it constant curvature} if the curvatures on the set of edges with positive limit weights are the same. Note that this concept differs from the definition of constant curvature on a manifold, where it is defined for the entire metric. A more precise definition is provided below.
\begin{definition}\label{def:constantcurvature}
Let $T=(V, E)$ be a tree equipped with an initial metric $w \in \mathbb{R}_{>0}^{E}$, and let the metric $w(t)$ evolve under the normalized Ricci flow. We say that the normalized Ricci flow \textbf{converges to a metric with constant curvature $\tau$} if there exists a metric $w(\infty) \in \mathbb{R}_{\geq 0}^{E}$ such that  $\lim\limits_{t \to \infty} w(t) = w(\infty)$ and under the metric $w(\infty)$ the curvature on each edge in $E^{+}:=\{e\in E(T):w_e(\infty)>0\}$ is $\tau.$

\end{definition}

In the project of studying Ricci flows on trees, we propose the following conjecture.  
\begin{conjecture}\label{conj:constant}
For a  tree with an initial metric, the normalized Ricci flow converges to a metric with constant curvature. 
\end{conjecture}  

We provide an affirmative answer to a special case that the flow converges  metric with constant curvature zero. 
\begin{theorem}\label{thm:limexist0}  Let $T=(V, E)$ denote a tree with an initial metric $w \in \mathbb{R}^{E}_{>0}$. The normalized Ricci flow converges to a metric with constant curvature zero only if the tree can be derived from a path $P=v_1-v_2-\ldots -v_n$ by adding pending vertices such that
\begin{enumerate}  
    \item the terminal vertex $v_1$ ($v_n$ resp.) connects to $d_{v_1}-1$ ( $d_{v_n}-1$ resp.) leaves, and
    \item for any other vertex $v_i$ of $P$, either $d_{v_i}=2$ or $v_i$ is incident to $d_{v_i}-2$ leaves. 
\end{enumerate}  
See e.g. Figure~\ref{fig:hh} for the structure of such a tree, we call such tree as caterpillar tree. 

For caterpillar tree, the limit of the unnormalized weight can be $0$, $\infty$ or other finite values that can be influenced by the initial metric; the limit of the normalized weight of leaves is always $0$.   

For non-caterpillar tree, there must be edges with an increase in unnormalized weight to infinity. 
\end{theorem}  

This theorem highlights the structural simplicity required for constant curvature zero. Intuitively, internal vertices with more internal edges create negative curvatures,  preventing the system from stabilizing to a flat metric. 

Since there is no known variational structure for the  Ricci flow, the proof of the convergence is a difficult task. The proof techniques presented in this paper rely on the simple structure of the trees, which determine distinct long-time behaviors of leaf edges and internal edges. We carefully study the structure of Ricci flow equations, and analyze the long time behaviors using analytic techniques combining with structural restriction of trees.


The structure of the paper is as follows: Section \ref{sec:pre} outlines the notation and foundational knowledge regarding curvature and the Ricci flow. In section \ref{sec:cur}, the curvature formula is introduced along with essential information about flows. Section \ref{subsec: unnorweight} establishes that the unnormalized weight on every edge monotonically grows and explores the  relationship between weights of adjacent edges. Section \ref{subsec: limitofcur} provides the necessary and sufficient conditions for curvature to converge to zero and to stay strictly positive or negative. Finally, Section \ref{subsec:zeorcur} examines the specific tree configurations that result in zero curvature on edges while maintaining strictly positive weights. 

\section{Preliminaries and Definitions}\label{sec:pre}
Let $G=(V, E, w)$ be a weighted graph on the vertex set $V$ and the edge set $E$ associated with a weight function $w:E\to (0, \infty)$ on the edge set $E$, we say $w$ defines a metric on $G$. For any two vertices $x, y$, we write $xy$ or $x\sim y$ to represent an edge $e=\{x, y\}$.  The length of a path is the sum of the edge lengths of the path; for any two vertices $u, v$,  the distance $d(x, y)$ is the length of a minimal weighted path among all paths that connect $u$ and $v$.
For any vertex $x\in V$, denote the neighbors of $x$ by $N(x)$ and by $d_x$ the degree of $x$.
Next, we recall the definition of the Ricci curvature defined on weighted graphs.

\begin{definition}\label{probabilitydistribution}
 Let $G=(V, E, d)$ be a graph with distance $d$. 
  Suppose that two probability distributions $\mu_1$ and $\mu_2$ on $V$  have finite support. A coupling between $\mu_1$ and $\mu_2$ is a probability distribution on $A$ on $V\times V$ with finite support so that
$$\sum\limits_{y \in V} A(x, y)=\mu_1(x) \ \text{and}\  \sum\limits_{x \in V} A(x, y)=\mu_2(y). $$

The transportation distance between two probability distributions $\mu_1$ and $\mu_2$ is defined as
\begin{align*}
W(\mu_1, \mu_2)=\inf\limits_{A} \sum\limits_{x, y\in V} A(x, y)d(x, y),\end{align*}

where the infimum is taken over all couplings $A$ between $\mu_1$ and $\mu_2$.
\end{definition}

Any coupling function provides a lower bound for the transport distance; the following definition can provide an upper bound for the transportation distance.
\begin{definition}\label{Lipschitz}
Let $G=(V, E, d)$ be a locally finite weighted graph. Let $f: V \to \mathbb{R}$. We say that $f$ is
1-Lipschitz if for any $x, y\in V$, 
$$f(x)-f(y)\leq d(x, y).$$
\end{definition}
 According to the duality theorem in linear optimization, as established by Kantorovich, the transportation distance can also be expressed as
\begin{align*}
W(\mu_1, \mu_2)=\sup\limits_{f} \sum\limits_{x\in V} f(x)[\mu_1(x)-\mu_2(x)],
\end{align*}
where the supremum is taken over all 1-Lipschitz functions $f$.

We will call $A\in V\times V$ satisfying the above infimum an optimal transportation plan and
call $f\in Lip(1)$ satisfying the above supremum an optimal Kantorovich potential
transporting $\mu_1$ to $\mu_2$.

\begin{definition} \cite{Ollivier,LLY,blhy}\label{def:alpharicci}
Let $G=(V, E, d)$ be a weighted graph where the distance $d$ is determined by the weight function $w$. Let $\gamma: \mathbb{R}_{+} \rightarrow \mathbb{R}_{+}$ represent an arbitrary one-to-one function.  For any $x, y\in V$ and $\alpha \in [0, 1]$,  the \textit{$\alpha$-Ricci curvature} $\kappa_{\alpha}$ is defined to be
$$\kappa_{\alpha}(x, y)=1-\frac{W(\mu_x^{\alpha}, \mu_y^{\alpha})}{d(x, y)},  $$
where the probability distribution $\mu_x^{\alpha}$ is defined as
\[\mu_x^{\alpha}(y)=
\begin{cases}
 \alpha,  & \text{if $y=x$}, \\
 \displaystyle (1-\alpha)\frac{\gamma(w_{xy})}{\sum\limits_{z\sim x} \gamma(w_{xz})},  & \text{if $y\sim x$},\\
 0,& \text{otherwise}.
\end{cases}\]

The Lin-Lu-Yau version of Ollivier Ricci curvature on vertices $(x, y)$ is defined as
\[\kappa(x, y)=\lim\limits_{\alpha\rightarrow 1} \frac{\kappa_{\alpha}(x, y)}{1-\alpha}. \]
\end{definition}



In the context of Ricci flow on graphs, the metric evolution is described in terms of the weights on the edges of a graph. Let $G$ be a graph with $m$ edges. The metric at any given time $t$ is represented by a vector $w(t) = \langle w_{e_1}(t), w_{e_2}(t), \dots, w_{e_m}(t) \rangle \in \mathbb{R}_{+}^m$, where each $w_{e_i}(t)$ is the weight associated with the edge $e_i$ at time $t$, and the parameter $t$ represents the time parameter. 

Given an initial metric $w_0 = \langle w_{e_1}(0), w_{e_2}(0), \dots, w_{e_m}(0) \rangle$, where each $w_{e_i}(0) > 0$, the unnormalized continuous Ricci flow on graphs is defined by the evolution of the weights $w_{e_i}(t)$ governed by a system of differential equations
\begin{equation*}
\begin{cases}
    w(0) = w_0,\\
    \frac{\partial w_e(t)}{\partial t} = -\kappa_e(t) w_e(t), \ \text{for all $e\in E(G)$},
\end{cases}
\end{equation*}
where $\kappa_e(t)$ is the curvature on edge $e$ at time $t$.




The normalized Ricci flow on graphs is a scaled version of the unnormalized Ricci flow. Given an initial metric $w_0=\langle w_{e_1}(0), w_{e_2}(0), \dots, w_{e_m}(0) \rangle$ such that  $\sum_{i=1}^{m} w_{e_i}(0) = 1$, we define a normalized version of the Ricci flow equations for graph:
\begin{equation*}
\begin{cases}
   w(0) = w_0,\\
    \frac{\partial w_e}{\partial t} = -\kappa_e w_e + w_e \sum\limits_{h \in E(G)} \kappa_h w_h.
\end{cases}
\end{equation*} 

Note that in order to maintain the invariance of the curvature $\kappa_e(t)$ under a metric scaling, it is preferable to consider the function $\gamma(x) = Ax^{a}$ for some real numbers $A$ and $a$.





Both the solutions of (\ref{eq:unnor_continuous})  and (\ref{eq:stochastic_continuous})  update the edge weight of $G$, and
the normalized flow differs from the unnormalized flow  only by rescaling in length of edges so that the total weight remains constant $1$, that is, $\sum_{i=1}^{m} w_{e_i}(t) = 1$ for any $t\geq 0$. 
For every edge $e$, we refer to the weight of edge $e$ after being updated by a solution of the normalized Ricci flow (\ref{eq:stochastic_continuous}) as $w_{e}(t)$, while for the unnormalized Ricci flow (\ref{eq:unnor_continuous}), if we write it as $\tilde{w}_{e}$.  Then,  the connection between the normalized and unnormalized flows is given by
\begin{align}\label{equ:unnor-nor}
 w_e(t)=\frac{\tilde{w}_e(t)}{\sum_{h\in E}\tilde{w}_h(t)}. 
\end{align}
To prevent symbol misuse, we use one symbol $w$ to represent the weight (normalized or unnormalized).

From these systems, we can see that the evolution of $w_e(t)$ depends on the instantaneous curvature $\kappa_e(t)$
and the interactions with other edge weights. This is a crucial aspect of how the Ricci flow dynamics unfold in discrete settings. 
In a previous study \cite{bailin}, we established conditions for the existence and uniqueness of global  solutions for Ricci flows on general graphs.


\begin{theorem}\label{thm:sol_continuous}
Let $\gamma:\mathbb{R}_{+} \rightarrow \mathbb{R}_{+}$ be a Lipschitz function over $[\delta, 1]$ for all $\delta > 0$. For any initial weighted graph $G$, by removing edges that violate the distance condition for any time,
there exists a unique solution $X(t)$, for all time $t\in [0, \infty)$,  to the system of ordinary differential equations in \eqref{eq:stochastic_continuous} (respectively \eqref{eq:unnor_continuous}). 
\end{theorem}

Given appropriate initial conditions, typically a positive initial weight $w_{uv}$ for each edge $uv$, the existence and uniqueness theorem ensures that there exists a unique solution $w_{uv}(t)$ defined for $t\geq 0$. However, this does not necessarily imply anything about long-time behavior or the limit of the solution as $t$ goes to infinity.  The solution limit may depend on additional conditions, such as the nature of $\kappa_{u v}(t)$, the initial conditions, and the structure of the graph. Therefore, we need to conduct further investigation into the specific behavior of the Ricci flow dynamics in the context of different types of graphs.

\subsection{ Ricci curvature on trees.}\label{sec:cur}

Let $T$ be a tree that is not a path or a star, where each vertex of the tree has a finite degree. In this context, an edge $uv$ in $T$ is referred to as a leaf edge if either vertex $u$ or $v$ is of degree $1$, while all other edges are considered as internal edges. A vertex $u$ is classified as a leaf vertex if its degree is $1$, as an internal vertex if its degree is at least $2$, and as a leaf-internal vertex if it is connected to a leaf edge. Let $d_u'$ be the number of leaves adjacent to the vertex $u$, and $d_u^{''}$ be the number of internal edges incident to $u$; we have $d_u=d'_u+d_u^{''}$. Throughout this paper, we assume that the tree $T$ is finite and connected.

Since there is exactly one path connecting two vertices in a tree,  it is easy to find an optimal Lipschitz function or an optimal transportation plan to calculate the Ricci curvature. 
Given an edge $uv$, at any time $t$, define a function $g$ on the vertex set $V(G)$ with 
$$g(x)=\begin{cases}
 w_{ux}+w_{uv} & \text{if}~~ x\sim u, x\neq v;\\
 w_{uv} & \text{if}~~ x=u\\
0  & \text{if}~~ x=v\\
 -w_{vy}  & \text{if}~~ y\sim v, y\neq u; \\
 0 & \text{otherwise}. 
\end{cases}
$$
It is straightforward to check that $g(x)$ serves as the Kantorovich potential for transport from $\mu_u^{\alpha}$ to $\mu_v^{\alpha}$. It can be shown, based on the proof of Lemma 3.2 in \cite{blhy}, that this function $ g $ serves as an optimal Lipschitz function.

For any vertex $u$, let \[D_u=\sum\limits_{v\sim u}\gamma(w_{uv}).\]  
By direct computation, one can derive the following formulas
 \begin{align*}
W(\mu_u^{\alpha}, \mu_v^{\alpha})&=\sum\limits_{x\in V} g(x)[\mu_u^{\alpha}(x)-\mu_v^{\alpha}(x)]\\
&=\sum\limits_{x\sim u,
x\neq v} (w_{ux}+w_{uv}) \Big((1-\alpha)\frac{\gamma(w_{xu})}{D_u} -0\Big) +w_{uv}\Big(\alpha- (1-\alpha)\frac{\gamma(w_{uv})}{D_v}\Big )\\&\quad \quad-\sum\limits_{y\sim v, y\neq u} w_{vy}\Big(0-(1-\alpha)\frac{\gamma(w_{vy})}{D_v}\Big)\\
&=(1-\alpha)\sum\limits_{x\sim u} (w_{ux}+w_{uv}) \frac{\gamma(w_{xu})}{D_u} -2(1-\alpha)w_{uv}\frac{\gamma(w_{uv})}{D_u}+w_{uv}\alpha \\
&\quad \quad - w_{uv}(1-\alpha)\frac{\gamma(w_{uv})}{D_v} + (1-\alpha)\sum\limits_{y\sim v} w_{vy}\frac{\gamma(w_{vy})}{D_v}-(1-\alpha)w_{vu}\frac{\gamma(w_{vu})}{D_v}.\\
\end{align*}
Then 
\begin{align*}
\kappa_{\alpha}(u, v)&=1-\frac{W(\mu_u^{\alpha}, \mu_v^{\alpha})}{w_{uv}}\\&=1-(1-\alpha)\sum\limits_{x\sim u} \frac{w_{ux}+w_{uv}}{w_{uv}}\frac{\gamma(w_{xu})}{D_u}+2(1-\alpha)\frac{\gamma(w_{uv})}{D_u} -\alpha \\&\quad \quad+(1-\alpha)\frac{\gamma(w_{uv})}{D_v} - (1-\alpha)\sum\limits_{y\sim v} \frac{w_{vy}}{w_{uv}}\frac{\gamma(w_{vy})}{D_v}+(1-\alpha)\frac{\gamma(w_{vu})}{D_v}, 
\end{align*}

Thus, 
\begin{align*}
\begin{split}
 \kappa(u, v)  &=1-\sum\limits_{x\sim u} \frac{w_{ux}+w_{uv}}{w_{uv}}\frac{\gamma(w_{xu})}{D_u}+2\frac{\gamma(w_{uv})}{D_u}+2\frac{\gamma(w_{uv})}{D_v}-\sum\limits_{y\sim v} \frac{w_{vy}}{w_{uv}}\frac{\gamma(w_{vy})}{D_v}\\
 &=1-\frac{1}{w_{uv}D_u}\sum\limits_{x\sim u} w_{ux}\gamma(w_{xu})-\sum\limits_{x\sim u} \frac{\gamma(w_{xu})}{D_u}+2\frac{\gamma(w_{uv})}{D_u}+2\frac{\gamma(w_{uv})}{D_v}\\&\quad\quad-\frac{1}{w_{uv}D_v}\sum\limits_{y\sim v} w_{vy}\gamma(w_{vy})\\
 &=-\frac{1}{w_{uv}D_u}\sum\limits_{x\sim u} w_{ux}\gamma(w_{ux}) +2\frac{\gamma(w_{uv})}{D_u}+2\frac{\gamma(w_{uv})}{D_v}-\frac{1}{w_{uv}D_v}\sum\limits_{y\sim v} w_{vy}\gamma(w_{vy})
 \end{split}
\end{align*}

\subsection{ The parameter function \text{$\gamma(x) = \frac{1}{x}$} and basic facts.}

Initially, we need to identify the kind of function $\gamma(x)$ that can assist in analyzing the behavior of $\kappa_{uv}$. Given the symmetry of $\kappa_{uv}$, it can be considered to comprise two parts associated with endpoints $u$ and $v$. The part associated with $u$ is denoted by $\kappa_{u\to v}$ and expressed by 
\begin{align}\label{equ:kutov}
   \kappa_{u\to v}= \frac{2w_{uv}\gamma(w_{uv})-\sum\limits_{x\sim u} w_{ux}\gamma(w_{ux})}{w_{uv}D_u}.
\end{align}

A natural choice is to use a simple, positive, and decreasing function such as $\frac{1}{x^{\alpha}}$
(where $\alpha>0$), as it controls the growth of the terms in the expression of $\kappa_{uv}$. In their investigation of the Euclidean framework of edge length dynamics involving Ricci curvature on graphs \cite{Gubser:2016htz}, the authors selected for $\alpha=2$. 
In our study, to maintain ease in mathematical computations, we choose $\alpha=1$. In particular, in situations where $d_u=1$, equation (\ref{equ:kutov}) consistently evaluates to $1$. Conversely, when $d_u=2$, equation (\ref{equ:kutov}) equates to $0$, thus simplifying the calculations in cases where the degree is $2$. Furthermore, as established in \cite{blhy}, with $\alpha=1$, the total Ricci curvature of all edges in a tree is $2$.

The choice of the parameter function $\gamma(x) = \frac{1}{x}$ is pivotal in driving the flow to constant curvature. 
Observe that the function $\gamma(x) = \frac{1}{x}$ is Lipschitz on the interval $[\delta, 1]$ for any $\delta > 0$. By substituting $\gamma(x) =  \frac{1}{x}$ into the expression for $\kappa_{uv}$, we derive a function $\kappa_{uv}w_{uv} : w \to R$ that remains Lipschitz continuous over the domain $(\delta, 1)^{E(T)}$ (refer to \cite{bailin}). Moreover, when $w_{uv}(t) \to 0$, it follows that $\kappa_{uv}(t)w_{uv}(t) \to 0$. Consequently, the function $\kappa_{uv}(t)w_{uv}(t) : w\to R$ is Lipschitz continuous on $(0, 1)^{E(T)}$, ensuring the uniqueness of the Ricci flow solution; see the justification in the proof of Theorem \ref{thm:sol_continuous} provided in \cite{bailin}. Thus, when setting $\gamma(x) = \frac{1}{x}$, we can confirm the existence of a unique long-time solution for the normalized Ricci flow (\ref{eq:stochastic_continuous}) (and thus for \eqref{eq:unnor_continuous}).

Furthermore, this choice $\gamma(x) = \frac{1}{x}$ has multiple key advantages in terms of curvature behaviors and edge weight evolution, which are explained in detail below:

\begin{itemize}
    \item \textbf{Bounded curvature:} The function \(\gamma(x) = \frac{1}{x}\) guarantees that the curvature remains bounded throughout the flow, even at the limit. By choosing this specific functional form, the Ricci flow equation effectively controls the growth of curvatures, preventing them from growing unbounded over time. 
    
    After simplification, the curvature on edge $uv$ at time $t$ can be calculated as follows:
\begin{align*}
\begin{split}
\kappa_{uv}(t)&=
      \frac{2-d_u}{w_{uv}(t)D_u(t)}+\frac{2-d_v}{w_{uv}(t)D_v(t)}\\
      &=\frac{2-d_u}{\sum\limits_{w\sim u}\frac{w_{uv}(t)}{w_{uw}(t)}}+\frac{2-d_v}{\sum\limits_{z\sim v}\frac{w_{uv}(t)}{w_{vz}(t)}}.
        \end{split}
          \end{align*}
      
By direct estimate, the curvature on every edge is restricted within an approximate range of $$4-d_u-d_v\leq \kappa_{uv}\leq 1. $$

It follows that the derivative function $\kappa'_{uv}(t)$ is also bounded, to see this, taking the derivative of $ \kappa_{uv}(t) $, we get \begin{align*}
  \kappa_{uv}^{\prime}&= \frac{d_u-2}{w_{uv}D_u}\left(\frac{\sum\limits_{x\sim u}\frac{\kappa_{ux}}{w_{ux}}}{ D_u}-\kappa_{uv}\right) + \frac{d_v-2}{w_{uv}D_v}\left(\frac{\sum\limits_{y\sim v}\frac{\kappa_{vy}}{w_{vy}}}{D_v}-\kappa_{uv}\right),
\end{align*} 
which is evidently bounded because the curvatures are bounded. 
Therefore,  $\kappa_{uv}(t)$ is a uniformly continuous function. 
      
    \item \textbf{The sum of curvatures is constant $2$.} The parameter function \(\gamma(x)\) also ensures that the total curvature of the tree remains constant and equal to 2, a critical property for the dynamics of the flow. 
    
    The sum of curvatures along all edges is as follows
\begin{align*}
     \sum_{uv\in E(T)} \kappa_{uv}(t) &=\sum_{uv\in E(T)} \frac{2-d_u}{w_{uv}(t)D_u(t)}+\frac{2-d_v}{w_{uv}(t)D_v(t)}
     \\&=\sum_{u\in V(T)} \frac{2-d_u}{D_u(t)}\sum_{v\sim u}\frac{1}{w_{uv}(t)}\\&=\sum_{u\in V(T)} (2-d_u)\\&=2.
\end{align*}
This is a discrete analog of the Gauss-Bonnet theorem  on a finite tree. These suggest that the Ricci flow system could potentially achieve equilibrium as $t$ goes to infinity.

What's more, 
 \begin{align}\label{equ:sumkfw}
 \begin{split}
    \sum_{h\in E(T)} \kappa_{h}w_{h} &=
    \sum_{u\in V(T)} \frac{(2-d_u)d_u}{D_u}\\
      &=\sum_{\text{leaf~}e}w_{e}+\sum_{u\in V, d_u\geq 2}\frac{(2-d_u)d_u}{D_u}.
            \end{split}
 \end{align}
It should be noted that the symbol $w$ in the above expression represents both the normalized and unnormalized weights, depending on the context.

\item \textbf{The presence of a leaf whose weight tends to zero.} 
When the edge weight is evolved using the unnormalized Ricci flow, one has
$$(\log w_{uv})^{\prime}=-\kappa_{uv},$$ Then 
$$(\log \prod\limits_{uv\in E(T)}w_{uv})^{\prime}=-\sum_{uv\in E(G)}\kappa_{uv}=-2, $$
which yields 
\begin{align}\label{producofweight}
\prod\limits_{uv\in E(T)}w_{uv}(t)=\Big(\prod\limits_{uv\in E(T)}w_{uv}(0)\Big)\times {\rm e}^{-2t}.    
\end{align}
That is, the total product of the unnormalized weight on all edges decreases exponentially fast to zero. 
 Moreover,  given that the curvature of the internal edges is always non-positive, there will be a leaf edge with an unnormalized weight approaching zero.
\end{itemize}

In summary, the function \(\gamma(x) = \frac{1}{x}\) is instrumental in achieving both the boundedness of the curvature and the preservation of the total curvature. These properties ensure the long-time convergence of the Ricci flow on trees and potentially contribute to the eventual convergence of the metric to a constant curvature metric.

\subsection{ An illustrative example of Ricci Flow on a tree.}

To  understand the dynamics of the continuous Ricci flow, we now provide an example.

\begin{example}\label{ex:interfinite}
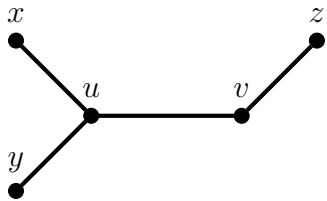
\begin{figure}[!h]
 \centering
\begin{tikzpicture}[scale=1, vertex/.style={circle, draw=black, fill=black}]
\node at (-1,1) [vertex] (v1) [label={$x$}, scale=0.5] {};
\node at (-1,-1) [vertex] (v2) [label={$y$}, scale=0.5] {};
\node at (0,0) [vertex] (v3) [label={$u$}, scale=0.5] {};
\node at (2,0) [vertex] (v4) [label={$v$}, scale=0.5] {};
\node at (3,1) [vertex] (v5) [label={$z$}, scale=0.5] {};
\draw [draw=black, ultra thick] (v1) -- (v3);
\draw [draw=black, ultra thick] (v2) -- (v3);
\draw [draw=black, ultra thick] (v3) -- (v4);
\draw [draw=black, ultra thick] (v4) -- (v5);
\end{tikzpicture}
\caption{A simple tree}
\end{figure}
The Ricci curvatures on the edges of tree in Figure \ref{ex:interfinite} are given by the following formulas:
\[
\kappa_{ux} = 1 - \frac{1}{w_{ux} D_u}, \quad
\kappa_{uy} = 1 - \frac{1}{w_{uy} D_u}, \quad
\kappa_{uv} = -\frac{1}{w_{uv} D_u}, \quad
\kappa_{vz} = 1,
\]
where \(D_u\) is defined as
\[
D_u = \frac{1}{w_{ux}} + \frac{1}{w_{uy}} + \frac{1}{w_{uv}}.
\]

We calculate the evolution of the edge weights under the unnormalized Ricci flow. For the edge \(vz\), we have
\[
\frac{\partial w_{vz}}{\partial t} = -w_{vz} \kappa_{vz} = -w_{vz}.
\]
This yields
\[
w_{vz}(t) = w_{vz}(0) e^{-t}.
\]

For the leaves \(ux\) and \(uy\), we have the following equations:
\[
\frac{\partial w_{ux}}{\partial t} = -w_{ux} \kappa_{ux} = -w_{ux} + \frac{1}{\frac{1}{w_{uv}} + \frac{1}{w_{ux}} + \frac{1}{w_{uy}}},
\]
and similarly,
\[
\frac{\partial w_{uy}}{\partial t} = -w_{uy} \kappa_{uy} = -w_{uy} + \frac{1}{\frac{1}{w_{uv}} + \frac{1}{w_{ux}} + \frac{1}{w_{uy}}}.
\]
Then we get $w_{ux}(t)-w_{uy}(t)=(w_{ux}(0)-w_{uy}(0))e^{-t}$.

Assuming that \(w_{ux}(t) = w_{uy}(t)\) and that \(w_{ux}(t) \leq w_{uv}(t)\) as \(t\) becomes large(which will demonstrate its validity in Proposition \ref{pro:leafinternal1}), we get the following inequality for the rate of change:
\[
-\frac{2}{3} w_{ux} < \frac{\partial w_{ux}}{\partial t} < -\frac{1}{2} w_{ux}.
\]
Thus, we have the following bounds for the edge weight \(w_{ux}(t)\):
\[
w_{ux}(0) e^{-\frac{2}{3} t} < w_{ux}(t) < w_{ux}(0) e^{-\frac{1}{2} t},
\]
which implies that \(w_{ux}(t) \to 0\) as \(t \to \infty\). Similarly, \(w_{uy}(t) \to 0\)  as \(t \to \infty\).

For the internal edge \(uv\), we calculate the evolution of its weight as follows:
\[
\frac{\partial w_{uv}}{\partial t} = -w_{uv} \kappa_{uv} = \frac{1}{\frac{1}{w_{uv}} + \frac{1}{w_{ux}} + \frac{1}{w_{uy}}}.
\]
This inequality gives the lower bound
\[
\frac{\partial w_{uv}}{\partial t} \geq \frac{1}{3} w_{ux}.
\]
On the other hand, we also have the upper bound
\[
\frac{\partial w_{uv}}{\partial t} \leq \frac{1}{2} w_{ux}.
\]
Therefore, the weight \(w_{uv}(t)\) satisfies the following inequality:
\[
0<w_{uv}(t) \leq w_{uv}(0)+ w_{ux}(0) -\frac{1}{2} w_{ux}(0) e^{-\frac{1}{2} t}.
\]
Thus, \(w_{uv}(t)\) is a growing function that is bounded from above, then \(\lim\limits_{t \to \infty} w_{uv}(t)\) exists and is finite. The behavior of \(w_{uv}(t)\) depends on the initial conditions. The plots of \(w_{uv}(t)\) in the subsequent figures illustrate this evolution.

\begin{figure}[H]
\centering
\includegraphics[scale=0.5]{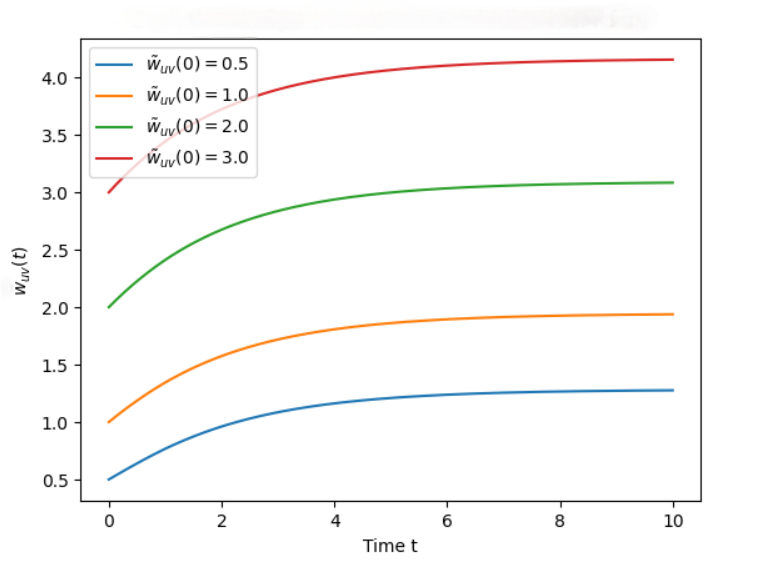}
\caption{Evolution of unnormalized weight on edge uv for different initial values}\label{fig:simpletreeaa}
\end{figure}

Using the relation (\ref{equ:unnor-nor}), the normalized weights
$w_{ux}(t),w_{uy}(t), w_{vz}(t)$ tend to zero and $w_{uv}(t)$ tends to $1$.

\end{example}



\subsection{Some lemmas}
The following lemmas and theorems will be used in the proof. 
 \begin{lemma}
[Barbalat's Lemma ]\label{barlem1}
Let $f: R \rightarrow R$ be a uniformly continuous function. Suppose that $\lim\limits_{t \rightarrow \infty} \int_0^t f(\tau) \mathrm{d} \tau$ exists and is finite. Then, $\lim _{t \rightarrow \infty} f(t)=0$
\end{lemma}
 For continuously differentiable functions, Barbalat’s lemma can be restated as
follows.

 \begin{lemma}
\label{barlem2}
If $f(t)$ has a finite limit as $t\to \infty$ and if $f^{\prime}(t)$ is uniformly continuous (or $f^{\prime\prime}(t)$ is bounded), then $\lim_{t\to\infty} f^{\prime}(t)=0$.  
\end{lemma}

 
\begin{lemma}[Gronwall lemma]\label{lem:Gronwall}
 Let $\phi:[a, T] \rightarrow \mathbb{R}$ be a nonnegative differentiable function for which there exists a constant $C$ such that
$$
\phi^{\prime}(t) \leq C \phi(t) \quad \text { for all } t \in[a, T].
$$
Then
$$
\phi(t) \leq e^{C t} \phi(a) \quad \text { for all } t \in[a, T].
$$
For a more general result,  assume that $\phi^{\prime}(t) \leq C(t) \phi(t)$ (where $C$ is a nonnegative summable function), then
$$
\phi(t) \leq \phi(a) \exp \left(\int_{a}^{t} C(\tau) d \tau\right).
$$

Moreover, if
$\phi^{\prime}(t)=C(t) \phi(t)+f(t), t \in [0, T]$ (where $C(t), f(t)$ are continuous functions),
then
$$
\phi(t)=\phi(a) e^{\int_a^t C(\rho) d \rho}+\int_a^t f(s) e^{\int_s^t C(\rho) d \rho} d s,~ t \in [0, T].
$$
\end{lemma}


\begin{lemma}\label{lem:halemma}
Let $f$ be a $C^1$ function on $[a, b]$. If $f(a)< 0$, and when  $f=0$  we have $df/dt< 0$, then $f(b)\leq 0$. 
\end{lemma}

\subsection{Outline of the proof for the results}

The interplay between the monotonic behavior of edge weights and the limiting behavior of curvatures under the Ricci flow is a central aspect of our analysis. The proof begins by establishing the asymptotic behavior of unnormalized weights, showing that for sufficiently large time, the weight on every edge exhibits monotonic behavior. This monotonicity holds without relying on specific knowledge about the curvature’s behavior, providing a robust framework for understanding the dynamics of the flow at the level of individual edges.

Subsequently, we demonstrate that the curvature on certain edges either converges to zero or remains strictly positive. This analysis begins with the leaf edges and then extends to internal edges,  avoiding the need for the precise information about the weight ratios of adjacent internal edges.  

With our understanding of the weight and curvature of leaves, as well as their connection to internal edges, we are able to describe the trees that attain zero constant curvature.

\section{Asymptotic behavior of weights under the Ricci flow} 
\label{subsec: unnorweight}
To establish the existence of a limit metric for the Ricci flow, we start by studying the asymptotic behavior of  unnormalized weight.


\begin{proposition}\label{pro:leafinternal1}
Let the metric $w$ of the tree $T$ evolve under the Ricci flow (\ref{eq:unnor_continuous}). 
Let $f$ represent an arbitrary internal edge. Then, we have

(ia) $\lim\limits_{t\to \infty} w_f> 0$ exists, being positive and possibly infinite;  (ib) if
 $\lim\limits_{t\to \infty} w_f$ is finite,  $\lim\limits_{t\to \infty} \kappa_f=0$. 
 
If $f$ is adjacent to leaves, let $u$ be the vertex connected to leaves, let $e$ represent a leaf incident to $u$, we have the following fact
 
 (ii) $w_e(t)<w_f(t)$ for large time $t\gg 1$.  

\end{proposition}

\begin{proof}
According to the Ricci flow, for any edge $f$, 
$$ \quad w_f(t)= w_f(0)e^{\int_0^t -\kappa_f(\tau) \, d\tau}. $$
Given that the initial condition $w_f(0)>0$, it follows that
$$w_f(t)> 0 \ \  \text{for all time $t$}.$$

For the internal edge $f=uv$, since $\kappa_f(t)\leq  0$ for all time $t$,  $w_f$ is always non-decreasing. Thus, $\lim\limits_{t\to \infty} w_f$ exists and $\lim\limits_{t\to \infty} w_f \geq w_f(0)$. 
 The result of (ia) follows. 
 
From the equation $w'_f = -\kappa_f w_f$, the second derivative can be derived as
\begin{align*}
    w^{\prime\prime}_f = -\kappa'_f w_f + \kappa^2_f w_f.
\end{align*}
If $\lim\limits_{t \to \infty} w_f$ is finite, by $\kappa'_f$ is bounded, $w^{\prime\prime}_f$ then remains finite. By Barbalat's Lemma \ref{barlem2}, it follows that the derivative $ w'_f$ tends to zero, implying that $\lim\limits_{t \to \infty} \kappa_f = 0$. Consequently, the result of (ib) holds.
 


For any leaf edge $e$ incident to $u$, we consider the  difference between their weights; by calculation, we have:
\begin{align*}
  w_f'(t)-w_e'(t) &=\kappa_ew_e-\kappa_fw_f\\
  &= w_e+\frac{2-d_u}{D_u}-(\frac{2-d_u}{D_u}+\frac{2-d_v}{D_v}) \\
  &=w_e+\frac{d_v-2}{D_v}.
\end{align*}
Since $f$ is not a leaf, $d_v\geq 2$. 
The term $w_e+\frac{d_v-2}{D_v}$ remains strictly positive throughout finite time periods, and can only approach zero at $t=\infty$.
Thus, $w_f(t)-w_e(t)$ is an increasing monotonic function, then $\lim\limits_{t \to \infty}(w_f(t)-w_e(t))$ exists or is infinity. If  $\lim\limits_{t \to \infty}(w_f(t)-w_e(t))=\infty$, clearly, the result of (ii) follows. If  $\lim\limits_{t \to \infty}(w_f(t)-w_e(t))$ is finite, then 
$$\lim\limits_{t \to \infty}w_e=\lim\limits_{t \to \infty}w_f(t)- \lim\limits_{t \to \infty}(w_f(t)-w_e(t))$$ exists or is infinity. 

Clearly, if $\lim\limits_{t \to \infty}w_e=0$, we are done. Otherwise, let $\lim\limits_{t \to \infty}w_e=c>0$ possibly $c=\infty$. Then  $w_f'(t)-w_e'(t)\geq c>0$, we then  have 
$w_f(t)-w_e(t)>w_f(0)-w_e(0)+ct>0$. 
The result of (ii) follows.

\end{proof}

Next, we further explore the development of weights $w_e(t)$ and $w_f(t)$, it is necessary to analyze how the weights of two leaves are correlated.

\begin{proposition}\label{pro:leaf=leaf}
Let the metric $w$ of the tree $T$ evolve under the  Ricci flow (\ref{eq:unnor_continuous}). 
If two leaves $e$ and $g$ are both incident to the same vertex $u$, then we have the following:

(i) If $w_e(0)\leq w_g(0)$, then $w_e(t)\leq w_g(t)$  for all $t>0$; conversely, if $w_e(0)>w_g(0)$
$w_e(t)>w_g(t)$ for all $t>0$.\\

(ii) $\lim\limits_{t\to \infty} (w_e(t)-w_g(t))=0$;\\

(iii) $\frac{w_e(t)}{w_g(t)}$ converges  to $1$ as $t\to \infty$.

\end{proposition}

\begin{proof}
Let $e$ and $g$ be two leaves incident to $u$. Then $d_u\geq 3$.  
Since
\begin{align*}
  w_e'(t)-w_g'(t) &=\kappa_gw_g-\kappa_ew_e\\
  &= w_g+\frac{2-d_u}{D_u}-(w_e+\frac{2-d_u}{D_u}) \\
  &=w_g-w_e, 
\end{align*}
we have   $$w_e(t)-w_g(t)=(w_e(0)-w_g(0))e^{-t}.$$ Results in (i) and (ii) follow. 

 Let $g$ denote the leaf with the least initial weight among all leaves sharing $u$.  It suffices to prove the ratio $\frac{w_e(t)}{w_g(t)}$ tends to $1$ for the result of (iii). 
We observe that 
$$\kappa_g=1-\frac{d_u-2}{w_gD_u}\leq 1-\frac{d_u-2}{d_u}, $$
which gives 
$$w'_g=-\kappa_gw_g\geq -(1-\frac{d_u-2}{d_u})w_g.$$

Therefore, for any other leaf $e$,
$$0\leq \frac{w_e(t)-w_g(t)}{w_g(t)}\leq \frac{(w_e(0)-w_g(0))e^{-t}}{w_g(0) e^{-(1-\frac{d_u-2}{d_u})t}}\leq \frac{w_e(0)-w_g(0)}{w_g(0)e^{\frac{d_u-2}{d_u}t}}.$$
Hence,
\[0\leq \frac{w_e(t)}{w_g(t)}-1\leq \Big(\frac{w_e(0)}{w_g(0)}-1\Big)  e^{\frac{2-d_u}{d_u}t}, \]
which indicates that
$$\lim\limits_{t\to \infty}\frac{w_e(t)}{w_g(t)}=1,$$
$\frac{w_e(t)}{w_g(t)}$ exponentially decay to $1$. 

We proved the result (iii). 

\end{proof}

The next statement is a consequence of the preceding propositions.
\begin{lemma}\label{lem:maxleaf}
    Let the metric $w$ of the tree $T$ evolve under the unnormalized Ricci flow (\ref{eq:unnor_continuous}). 
Let $e$ represent a leaf edge incident to vertex $u$ where $d_u\geq 3$.  If $e$ attains the maximum initial value $w_e(0)$, then $\kappa_e(t)\geq \kappa_g(t)$ for any other leaf $g$ incident to $u$, and  $\kappa_e(t)> 1-\frac{d_u-2}{d_u'}$ for $t\in [0, \infty)$. 
\end{lemma}

\begin{proof}
Observing the formula for leaf $e$, we have $$\kappa_e=1-\frac{d_u-2}{w_eD_u}=1-\frac{d_u-2}{\sum\limits_{\text{leaf } g} \frac{w_e(t)}{w_g(t)}+\sum\limits_{\text{internal } f} \frac{w_e(t)}{w_f(t)}}.$$
Clearly, if $w_e(t)\geq w_g(t)$ for $t\geq 0$, $\kappa_e(t)\geq \kappa_g(t)$.  The first assertion follows from 
the result  (i) of Proposition \ref{pro:leaf=leaf}.  
Besides, we have 
\[
 \sum_{\text{leaf } g} \frac{w_e(t)}{w_g(t)} \geq d_u'
\]
and 
\[\sum_{\text{internal } f} \frac{w_e(t)}{w_f(t)}>0, \]
It follows that for all time $t$, 
 $$\kappa_e(t)> 1-\frac{d_u-2}{d_u'}. $$

\end{proof}

Next, we study the limit behavior of the unnormalized weights on leaves.
\begin{proposition}\label{pro:leavesdecrefaster}
Let the metric $w$ of the tree $T$ evolve under the Ricci flow \eqref{eq:unnor_continuous}. 
Let $e$ represent a leaf edge
incident to vertex $u$. 
Then

(i) If there exists an internal edge $f\sim e$ such that $\lim\limits_{t\to \infty} w_f <\infty$, then $\lim\limits_{t\to \infty} w_e=0$;

(ii) $\lim\limits_{t\to \infty} w_e$ exists, being possibly infinite. 

\end{proposition}

\begin{proof}
Firstly, consider  the case when $\lim\limits_{t\to \infty} w_f$ is finite for some internal edge $f$, which implies, according to part (ib) of Proposition \ref{pro:leafinternal1}, that $\lim\limits_{t\to \infty} \kappa_f=0$.  Based on previous results, it follows that leaf $e$ is an edge such that $\frac{w_f}{w_e}\to 0$, then $w_e$ approaches zero. The results of (i) and (ii) follow.

Next, consider the case when $\lim\limits_{t\to \infty} w_f=\infty$ for every internal edge $f\sim u$. 
There are two cases to consider. 
\begin{itemize}
    \item[Case 1:] 
If $w_e(t)$ is bounded on $(0, \infty)$, then
this implies that for any another leaf incident to $u$, the weight is bounded according to Proposition \ref{pro:leaf=leaf}. 
Note that in this case, $\lim\limits_{t\to \infty}\frac{w_e(t)}{w_f(t)}=0$  for every internal edge $f\sim e$, Thus, 
 \[
\lim_{t \to \infty} \sum_{\text{internal } f\sim u} \frac{w_e(t)}{w_f(t)} = 0.
\]
Then, we get,
$$\lim\limits_{t\to \infty}\kappa_e= 1-\frac{d_u-2}{d_u'},$$
which indicates that the sign of $w'_e$ depends on the value of $d_u'$ for large time $t$. 
Notice that if $d_u'\leq d_u-3$, $\kappa_e$ would tend to a value that is at most $\frac{-1}{d_u-3}$, suggesting that $w_e(t)$ would tend to infinity. 
Therefore, either $d_u'=d_u-1$, implying $\kappa_e$ tends to $\frac{1}{d_u-1}$, then $w_e(t)$ would tend to zero; or 
$d_u'=d_u-2$, resulting in $\lim\limits_{t\to \infty}\kappa_e=0$. 

We then consider this case $\lim\limits_{t\to \infty}\kappa_e=0$. 
By  (i) in Proposition \ref{pro:leaf=leaf},  let leaf $e$ attain the maximum initial weight among leaves sharing $u$, for any $t$,
then we have 
\begin{align*}
 \kappa_e(t)
     & > 1 - \frac{d_u - 2}{d_u - 2}> 0.
\end{align*}
Thus,  the derivative $w'_e=-\kappa_ew_e$ is negative and then $w_e(t)$ is monotonically decreasing and, being limited below by zero, it converges to a bounded value.

As a result, $\lim\limits_{t\to \infty}w_e$ exists. This holds for any other leaf $g$ where $g\sim e$, according to result (ii) in Proposition \ref{pro:leaf=leaf}.

\item[Case 2:] 
Next, we assume that $w_e(t)$ is unbounded on $(0, \infty)$.  Note that if such an edge $e$ exists, then for any other leaf incident to the vertex $u$, their unnormalized weights must also be unbounded according to Proposition \ref{pro:leaf=leaf}. We will show that $w_e(t)$ necessarily tends to infinity using the following claim:

\textbf{Claim}: \emph{If there exists a time sequence for which $w_e(t)$ remains bounded, then $d_u' = d_u - 2$.}

Suppose, for contradiction, that $\kappa_e(t)$ must oscillate, switching between positive and negative values infinitely often, which would require it to cross zero infinitely many times. Since the derivative of $w_e(t)$ is defined everywhere, consider the sequence of local minima of $w_e(t)$, denoted $\{w_e(t_i)\}_{i \geq 1}$. By assumption, $w_e(t_i)$ does not tend to infinity.

Thus, there exists a bounded subsequence $\{w_e(t_{i_m})\}$. Now consider any internal edge $f$ connected to $u$. Since $w_f(t)$ diverges to infinity, we obtain
\[
\lim_{m \to \infty} \frac{w_e(t_{i_m})}{w_f(t_{i_m})} = 0.
\]
At these local minimas, we also have
\[
\kappa_e(t_{i_m}) = 0.
\]
Since for leaf $g\sim e$, we have 
\[
\lim_{t \to \infty} \sum_{\text{leaf } g} \frac{w_e(t)}{w_g(t)} = d_u'.
\]
Using this and $\kappa_e(t_{i_m}) = 0$, we derive
\[
0 = \lim_{m \to \infty} \kappa_e(t_{i_m}) = 1 - \frac{d_u - 2}{d_u' + 0},
\]
which implies $d_u' = d_u - 2$.

Without loss of generality, assume $e$ is the leaf with the maximum initial weight among all leaves sharing $u$. Similarly, by Lemma \ref{lem:maxleaf},  we have
\[
\kappa_e(t) > 1 - \frac{d_u - 2}{d_u - 2} > 0.
\]
This implies that $w_e(t)$ is strictly decreasing and the situation in which it remains unbounded without approaching infinity is not possible. 

Consequently, in the final situation, $w_e(t)$ approaches infinity. This concludes the proof.
\end{itemize}
\end{proof}

The following proposition asserts that if $w_e\to \infty$, then $w_e$ must monotonically increase to infinity.

\begin{proposition}\label{pro:ke<0always<0}
Let the metric $w$ of T evolve under the Ricci flow \eqref{eq:unnor_continuous}. 
For a leaf $e$, if $w_e(t)$ tends to infinity,  then there exists a time $t_0$ such that $\kappa_e(t)< 0$ for all $t\geq t_0$. That is,  $w_e(t)$ is monotonically increasing for $t\geq t_0$. 
\end{proposition}

\begin{proof}
 Let $e\sim u$,  clearly $d_u\geq 3$. By the formula $\kappa_e=1-\frac{d_u-2}{w_eD_u}$,  derivative of $\kappa_e$ is
\begin{align}\label{equ:deriofke}
 \kappa_e^{\prime}(t)=(1-\frac{d_u-2}{w_e D_u})^{\prime}=\frac{d_u-2}{w_e^2 D_u^2}(D_u'w_e+D_uw_e')=\frac{d_u-2}{w_e D_u}(\frac{\sum\limits_{v\sim u}\frac{\kappa_{uv}}{w_{uv}}}{\sum\limits_{v\sim u}\frac{1}{w_{uv}}}-\kappa_e).   
\end{align}
where $v$ ranges over all neighbors of $u$.

Given that $w_e(t)$ tends to infinity, there exists a sufficiently large time $t_0 > 0$ where $w_e^{\prime}(t_0) > 0$. Assume $w_e(t)$ does not have a monotonic increase on the interval $[t_0, \infty)$, then there exists $t_1>t_0$ such that $w_e^{\prime}(t_1) < 0$. By the continuity, there exists a point $a$ in $(t_0, t_1)$ such that $w_e^{\prime}(a) = 0$. Given that $w_e(t)$ is always positive, we have that $\kappa_e(t_0) < 0$ and $\kappa_e(a) = 0$.

Next,   at point $a$, we can derive $w_e(a)=\frac{d_u-2}{D_u(a)}$, and it follows that 
\begin{align*}
   D_u(a) \kappa_e^{\prime}(a)&=\sum\limits_{v\sim u}\frac{\kappa_{uv}(a)}{w_{uv}(a)}\\
    &=\sum\limits_{\text{leaf } g}\frac{\kappa_{g}(a)}{w_{g}(a)} +\sum\limits_{\text{internal } f}\frac{\kappa_{f}(a)}{w_{f}(a)}.
\end{align*}
Notice that for any internal edge $uv$ and finite time $a$, we have $\kappa_{uv}(a)<0$.
If $d_u'=1$, there is only one leaf $e$,  apparently  $\kappa_e^{\prime}(a)<0$.  
Let $d_u'>1$. W.l.o.g., we assume that $e$ attains the maximum initial weight among all leaves sharing $u$. 
Then for any leaf edge $g\sim e$, $\kappa_g(a)\leq 0$ according to Lemma \ref{lem:maxleaf}.
Since $D_u(t)$ is always positive, we conclude that  $\kappa_e^{\prime}(a)<0$ according to the formula (\ref{equ:deriofke}).

Given $\kappa_e(t_0)<0$ and $\kappa_e^{\prime}(a)<0$ whenever $\kappa_e(a)=0$, Lemma \ref{lem:halemma} implies $\kappa_e(t)\leq 0$ for all $t>t_0$, which contradicts the assumption $w_e^{\prime}(t_1) < 0$.


Hence, for every leaf $e$, if its unnormalized weight grows to infinity, there exists some time $t_0$ such that $w_e^{\prime}(t)\geq  0$ for all $t>t_0$. The result follows.


\end{proof}

\section{Asymptotic behavior of curvature}\label{subsec: limitofcur}
Based on previous knowledge on  unnormalized weights, we examine the criteria under which curvature approaches zero, or stays strictly positive or negative. 

\begin{proposition}\label{Pro:k_etendto0}
  Let the metric $w$ of T evolve under the Ricci flow \eqref{eq:unnor_continuous}. Consider $e$ as a leaf edge connected to vertex $u$ where $d_u \geq 3$. Consequently, the following statements are valid.
  
   (i)~ $d_u'=d_u-1$ if and only if $\lim\limits_{t\to \infty}\kappa_e(t)=\frac{1}{d_u-1}$;

  (ii) $\lim\limits_{t\to\infty} \kappa_e(t)=0$ if and only if $d_u'=d_u-2$. 

\end{proposition}

\begin{proof}
WLOG, let $w_e(0) \geq w_g(0)$ for any leaf $g$ incident to $u$. 
When $d_u'=d_u-1$, according to Lemma \ref{lem:maxleaf}, there exists $t_0>0$ such that, 
$\kappa_e(t)\geq 1-\frac{d_u-2}{d_u-1}=\frac{1}{d_u-1}$  for $t\geq t_0$, which gives $$w_e(t) \leq  w_e(t_0)e^{-\frac{1}{d_u-1}t}. $$ Then $\lim\limits_{t\to \infty}w_e(t)=0$. It follows  $$\lim\limits_{t\to \infty}\kappa_e(t)= 1-\frac{d_u-2}{d_u-1+0}=\frac{1}{d_u-1}.$$ 

Conversely, assume that $\lim\limits_{t\to \infty}\kappa_e(t)=\frac{1}{d_u-1}, $ it is easy to see that  $\lim\limits_{t\to \infty}w_e(t)=0, $ therefore, $$\lim\limits_{t\to \infty}\frac{w_e}{w_f}(t)=0,  $$
for every internal edge $f\sim u$. 
Thus, $d'_u=d_u-1.$
We proved the result (i).

Assume $\lim\limits_{t\to \infty}\kappa_e(t)= 0$, then obviously we have 
$$\lim\limits_{t\to \infty}\sum\limits_{\text{internal } f\sim e}\frac{w_e}{w_f}=d_u-2-d_u'. $$
One can get the second derivative 
\begin{align*}
   \Big(\sum\limits_{\text{internal} ~f\sim e}\frac{w_e}{w_f}\Big)^{\prime \prime} &= \sum\limits_{\text{internal} f\sim e}\Big(\frac{w_e}{w_f}(\kappa_f-\kappa_e) \Big)^{\prime}
   \\&= \sum\limits_{\text{internal }f\sim e}\frac{w_e}{w_f}\Big((\kappa_f-\kappa_e )^2+ \kappa_f'-\kappa_e'\Big),
\end{align*}
which is bounded, then  by Barbalat's Lemma \ref{barlem2},  the first derivative satisfies $$\lim\limits_{t\to \infty}\sum\limits_{\text{internal } f\sim e}\frac{w_e}{w_f}(\kappa_f-\kappa_e)=0.  $$
Hence, 
$$\lim\limits_{t\to \infty}\sum\limits_{\text{internal } f\sim e}\frac{w_e}{w_f}\kappa_f =0.$$
As the curvature $\kappa_f$ remains strictly negative for each $f$, It follows that for any given internal edge $f$,
\[
\lim_{t \to \infty} \frac{w_e}{w_f} \kappa_f = 0.
\]
Moreover, as both $\frac{w_e}{w_f}$ and $\kappa_f$ are bounded, if $ \frac{w_e}{w_f} \not\to 0$
along some time sequence, then necessarily $\kappa_f$ must converge to zero along that same sequence.

Denote $f = uv$, and note that
\[
\kappa_f = - \frac{d_u - 2}{\sum\limits_{ux} \frac{w_f}{w_{ux}}} 
           - \frac{d_v - 2}{\sum\limits_{vy} \frac{w_f}{w_{vy}}}.
\]

Observe that if $\kappa_f$ approaches zero along a time sequence, then 
$\frac{w_f}{w_{ux}}$ must diverge to infinity for at least one incident edge $ux$.  
By Proposition~\ref{pro:leafinternal1} together with result of (iii) of Proposition~\ref{pro:leaf=leaf}, this edge can be identified with the leaf $e$, and hence 
$\frac{w_e}{w_f}$ must converge to zero along the same sequence.

Consequently,
\[
\lim_{t \to \infty} \frac{w_e}{w_f} = 0,
\] 
which implies
\[
\lim_{t \to \infty} \kappa_e = 1 - \frac{d_u - 2}{d'_u + 0} = 0,
\] 
and hence
\[
d_u' = d_u - 2.
\]

On the other hand, suppose that $d_u' = d_u - 2$.  
By Lemma \ref{lem:maxleaf},  we have $\kappa_e(t) >0$ for large time $t$, which implies that $w_e(t)$ is decreasing.  
Consequently, for every internal edge $f \sim e$, the ratio $\frac{w_e}{w_f}$ is decreasing since the denominator is non-decreasing.  
It follows that
\[
\lim_{t \to \infty} \frac{w_e}{w_f} \quad \text{exists}.
\]

By Barbalat's Lemma, we then obtain for the derivative
\[
\lim_{t \to \infty} \frac{w_e}{w_f} \, (\kappa_f - \kappa_e) = 0.
\]

For this identity to hold, if 
\[
\lim_{t \to \infty} \frac{w_e}{w_f} > 0,
\]
then necessarily
\[
\lim_{t \to \infty} (\kappa_f - \kappa_e)(t) = 0.
\]
Moreover,
\[
\lim_{t \to \infty} \frac{w_f}{w_e} 
= \frac{1}{\lim_{t \to \infty} \frac{w_e}{w_f}} < \infty,
\]
which implies that $\frac{2-d_u}{w_f D_u}$ does not vanish along any time sequence, and hence neither does $\kappa_f$.  
On the other hand, since $\kappa_e$ remains positive for sufficiently large $t$, this contradicts the fact that $(\kappa_f - \kappa_e)$ tends to zero.  

Therefore, $\lim\limits_{t \to \infty} \frac{w_e}{w_f}=0$ for every internal edge, it follows that
\[
\lim_{t \to \infty} \kappa_e 
= 1 - \frac{d_u - 2}{d_u - 2} = 0.
\]

\end{proof}

\begin{corollary}\label{cor:wefinite}
 Let the metric $w$ of tree T evolve under the Ricci flow \eqref{eq:unnor_continuous}. 
Let $e=uv$ be a leaf with $d_u\geq 2$. Then the unnormalized weight $w_e(t)$ is bounded if and only if 
$d'_u =d_u - 2$ or $d'_u =d_u - 1$.

\end{corollary}
\begin{proof}
By Proposition~\ref{pro:leavesdecrefaster}, the $\lim_{t \to \infty} w_e(t)$ exists.  
Suppose that $\lim_{t \to \infty} w_e(t) < \infty$.  
By Barbalat's Lemma, the derivative 
\[
w'_e(t) = -\kappa_e(t)\, w_e(t)
\] 
must then converge to zero.  
Thus, if  $\kappa_e(t)$ does not approach zero along at least one time sequence, then $w_e(t) \to 0$ along this sequence.  
Since the limit of $w_e(t)$ exists, we conclude that in fact
\[
w_e(t) \to 0 \quad \text{as } t \to \infty.
\]
Then, for each internal edge $f \sim e$, the ratio $\frac{w_e}{w_f}$ tends to zero as $t \to \infty$. Assume $d'_u \leq d_u - 3$, then we have 
\[
\kappa_e \to 1 - \frac{d_u - 2}{d'_u+0} \leq -\frac{1}{d_u - 3}.
\]
In particular, the derivative $w'_e(t)$ remains positive for sufficiently large $t$, contradicting the fact that $w_e(t) \to 0$.

Thus, $d'_u$ is at least $d_u-2$ and $\kappa_e(t)$ approaches either $\frac{1}{d_u-1}$ or zero. 

We then prove the other direction. 
The case $d'_u= d_u-1$ is clear and $\lim\limits_{t\to \infty}w_e(t)=0$. 
Let $d'_u= d_u-2$, we have $\kappa_e\to 0$.    W.l.o.g. let $e$ be the leaf with the largest initial weight among all leaves sharing vertex $u$. According to Lemma \ref{lem:maxleaf}, we have for all $t>0$, $$\kappa_e(t)> 1-\frac{d_u-2}{d_u'}> 1-\frac{d_u-2}{d_u-2}>0.$$ Thus, the derivative $w'_e(t)$ is negative,  the unnormalized weight $w_e(t)$ decreases. Then $\lim\limits_{t\to \infty}w_e(t)$ is finite. 


\end{proof}

In the following, we continue to establish a bound for $\kappa_e$ in other cases.
\begin{proposition}\label{pro:kenegative}
Let the metric $w$ of T evolve under Ricci flow \eqref{eq:unnor_continuous}. 
For a leaf $e$ incident to a vertex $u$,  $d_u'\leq d_u-3$, 
then $\kappa_e$ is strictly negative for large time $t$ and 
\begin{align}\label{equ:kenegative}
\kappa_e\leq 1-\frac{d_u-2}{d_u'+\frac{(d''_u-2)d''_u}{(d''_u-2+d_u)}}. 
\end{align}

That is, 
$$\kappa_e\leq -\frac{d'_u}{d_u}. $$
\end{proposition}

\begin{proof}
Let  $d'_u\leq  d_u-3$, we have $\lim\limits_{t\to \infty}w_e=\infty$ by Corollary \ref{cor:wefinite}. Then there exists a time $t_0$, such that $\kappa_e(t_0)<0$.

To demonstrate the inequality (\ref{equ:kenegative}), it is sufficient to show that 
$$\sum\limits_{\textbf{internal }f\sim u}\frac{w_e}{w_f}(t)> \frac{(d''_u-2)d''_u}{(d''_u-2+d_u)} \ \text{implies that ~}\kappa_e^{\prime}(t)< 0.  $$ 

W.l.o.g., let $e$ attain the maximum initial weight among all leaves sharing $u$. 
Consider the derivative expression
\begin{align*}
\kappa_e^{\prime}(t) = (1-\kappa_e)\left(\frac{\sum\limits_{v \sim u} \frac{\kappa_{uv}}{w_{uv}}}{\sum\limits_{v \sim u} \frac{1}{w_{uv}}} - \kappa_e \right). 
\end{align*}
To show $\kappa'_e\leq 0$, it is equivalent to prove that 
\[\sum\limits_{v \sim u} \frac{\kappa_{uv}}{w_{uv}} \leq  \sum\limits_{v \sim u} \frac{\kappa_e}{w_{uv}},\]
which is equivalent to
\[\sum\limits_{\textbf{leaf } g} \frac{\kappa_{g}}{w_{g}} + \sum\limits_{\textbf{internal } f} \frac{\kappa_{f}}{w_{f}}\leq \sum\limits_{\textbf{leaf } g} \frac{\kappa_{e}}{w_{g}} + \sum\limits_{\textbf{internal } f} \frac{\kappa_{e}}{w_{f}}.\]
Given that $w_g\leq w_e$ for each leaf $g$,  we have $\kappa_g(t) \leq \kappa_e(t)<0$ and 
$\frac{\kappa_e}{w_g}\leq \frac{\kappa_e}{w_g}\leq \frac{\kappa_e}{w_e}.$ Then
it suffices to show:
\[\sum\limits_{\textbf{internal } f} \frac{\kappa_{f}}{w_{f}} \leq \sum\limits_{\textbf{internal } f} \frac{\kappa_{e}}{w_{f}}.\]
Since for every internal $f=uv$, $\kappa_f \leq \frac{2-d_u}{w_fD_u}$, where the equality holds if and only if $d_v=2$.

It suffices to prove
\[\frac{2-d_u}{D_u}\sum\limits_{\textbf{internal } f} \frac{1}{w^2_{f}}\leq \left(1 + \frac{2-d_u}{w_eD_u}\right) \sum\limits_{\textbf{internal } f} \frac{1}{w_{f}},\]
which is equivalent to:
\[\frac{2-d_u}{D_u}\sum\limits_{\textbf{internal } f} \frac{1}{w^2_{f}} - \frac{2-d_u}{w_eD_u} \sum\limits_{\textbf{internal } f} \frac{1}{w_{f}} \leq \sum\limits_{\textbf{internal } f} \frac{1}{w_{f}}.\]
Using the Cauchy-Schwartz inequality, we have \[\sum\limits_{\textbf{internal }f}\frac{1}{w^2_{f}} \geq \frac{1}{d''_u}\left(\sum\limits_{\textbf{internal }f}\frac{1}{w_{f}}\right)^2, \] where equality holds if and only if for all internal edge $f\sim u$, $w_f(t)$ are the same.

Then, it is necessary to show
\[
\frac{2-d_u}{D_u}\frac{1}{d''_u}\left(\sum\limits_{\textbf{internal }f}\frac{1}{w_{f}}\right)^2 - \frac{2-d_u}{w_eD_u}\sum\limits_{\textbf{internal }f}\frac{1}{w_{f}}\leq  \sum\limits_{\textbf{internal }f}\frac{1}{w_{f}}. \]
By dividing both sides by $\sum\limits_{\textbf{internal }f}\frac{1}{w_{f}}$, we derive
\[
\frac{2-d_u}{D_u}\frac{1}{d''_u}\sum\limits_{\textbf{internal }f}\frac{1}{w_{f}}- \frac{2-d_u}{w_eD_u} \leq 1. \]
Multiplying both sides by $D_u$, we have 
\[
\frac{2-d_u}{d''_u}\sum\limits_{\textbf{internal }f}\frac{1}{w_{f}}- \frac{2-d_u}{w_e} \leq D_u. \]
In other words,
\[
\frac{2-d_u}{d''_u}\sum\limits_{\textbf{internal }f}\frac{1}{w_{f}}- \frac{2-d_u}{w_e} \leq \sum\limits_{\textbf{internal }f}\frac{1}{w_{f}}+\sum\limits_{\textbf{leaf }g}\frac{1}{w_{g}}.\]
Since $\frac{1}{w_{e}}\leq \frac{1}{w_{g}}$ for every leaf $g$,  
the above inequality is equivalent to 

\[
\frac{2-d_u}{d''_u}\sum\limits_{\textbf{internal }f}\frac{1}{w_{f}}- \frac{2-d_u}{w_e} \leq \sum\limits_{\textbf{internal }f}\frac{1}{w_{f}}+\frac{d'_u}{w_{e}}.\]

That is,  
\[
(\frac{2-d_u}{d''_u}-1)\sum\limits_{\textbf{internal }f}\frac{1}{w_{f}}\leq \frac{2-d_u+d'_u}{w_e} .\]
Multiplying by $-w_e$ on both sides, 
we get 
\[
(1-\frac{2-d_u}{d''_u})\sum\limits_{\textbf{internal }f}\frac{w_e}{w_{f}}\geq d''_u-2.\]
Multiplying $d''_u$ by both sides, we get 
\[
(d''_u-2+d_u)\sum\limits_{\textbf{internal }f}\frac{w_e}{w_{f}}\geq (d''_u-2)d''_u.\]

Using $d_u=d'_u+d''_u$, we have 
$$1-\frac{d_u-2}{d_u'+\frac{(d''_u-2)d''_u}{(d''_u-2+d_u)}}=-\frac{d'_u}{d_u}.$$

The proof is complete.

\end{proof}

For the internal edge, we have the following fact. 
\begin{corollary}\label{pro:kf=0}
   Let the metric $w$ of tree T evolve under the Ricci flow \eqref{eq:unnor_continuous}. 
Let $f=uv$ be an internal edge.  If $d_u=2$ , or $d'_u\geq d_u-2$ and $d_v=2$,  or $d'_v\geq d_v-2$, then $\lim\limits_{t\to \infty}\kappa_f(t)=0$. 
\end{corollary}
\begin{proof}
    According to the proof of Proposition \ref{Pro:k_etendto0}, if  $d'_u\geq d_u-2$, as well as $d'_v\geq d_v-2$, then there exist leaves attached to $u$ and $v$ such that the ratio of their unnormalized weights to the unnormalized weight of the edge $f$ approaches zero. 
Consequently, $\kappa_f$ tends to zero as $t$ approaches infinity as both   $\kappa_{u\to f}\to 0$ and $\kappa_{v\to f}\to 0$ hold.
\end{proof}

\textbf{Proof of the Main Theorem \ref{thm:foundproperties}}

\begin{proof}
The existence of $\lim\limits_{t\to \infty}w_{uv}(t)$ in the main Theorem \ref{thm:foundproperties} is proved by combining the results from  propositions \ref{pro:leafinternal1} and \ref{pro:leavesdecrefaster}.

\end{proof}

\section{The limit metric with constant curvature zero}\label{subsec:zeorcur}

In this section,  we characterize the structure of trees for which the Ricci flow converges to a metric with constant curvature zero, as stated in Theorem \ref{thm:limexist0}.
For the structure of such a tree, refer to Figure \ref{fig:hh}.

\begin{figure}[!h]
\begin{center}
\begin{tikzpicture}[
  every node/.style={circle, draw, fill=white, minimum size=0.3cm},
  level distance=1.5cm,
  leaf/.style={circle, draw, fill=black, minimum size=0.3cm}
]

\node (v1) at (0,0) {};
\node (v2) at (2,0) {};
\node (v3) at (4,0) {};
\node (v4) at (6,0) {};
\node (v5) at (8,0) {};
\node (v6) at (10,0) {};

\draw (v1) -- (v2);
\draw (v2) -- (v3);
\draw (v3) -- (v4);
\draw (v4) -- (v5);
\draw (v5) -- (v6);
\foreach \i in {1,2,3} {
    \node[leaf] (l1\i) at (-1,\i) {};
    \draw (v1) -- (l1\i);
}

\foreach \i in {1,2} {
    \node[leaf] (l2\i) at (1,\i) {};
    \draw (v2) -- (l2\i);
}

\foreach \i in {1,2, 3} {
    \node[leaf] (l3\i) at (3,\i) {};
    \draw (v3) -- (l3\i);
     \draw (v3) -- (l3\i);
}

\foreach \i in {1,2} {
    \node[leaf] (l4\i) at (5,\i) {};
    \draw (v4) -- (l4\i);
}

\foreach \i in {1,2,3} {
    \node[leaf] (l6\i) at (9,\i) {};
    \draw (v6) -- (l6\i);
}

\end{tikzpicture}
\caption{Caterpillar tree derived from a path by adding pending vertices}\label{fig:hh} 
\end{center}
\end{figure}
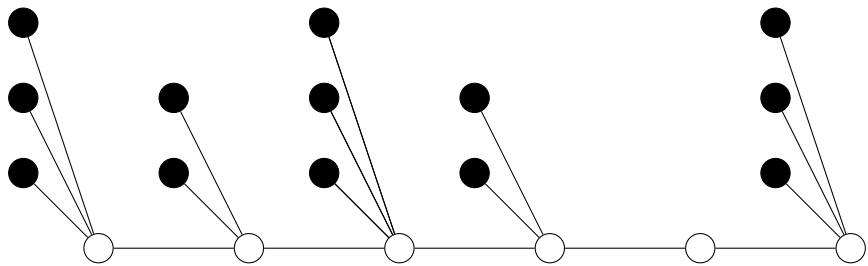

\begin{definition}[Caterpillar Trees]\label{caterpi}
A \emph{caterpillar tree} is a tree $T=(V,E)$ that contains a path 
$P = (v_1,v_2,\dots,v_{\ell})$, called the \emph{spine}, such that 
\[
V = V(P) \cup L,
\]
where $L$ is the set of leaves of $T$. More precisely:
\begin{itemize}
    \item If $v_i$ is an internal vertex of $P$ (i.e.\ $2 \leq i \leq l-1$), then $v_i$ is adjacent to exactly $d_{v_i}-2$ leaves in $L$.
    \item If $v_i$ is an endpoint of $P$ (i.e.\ $i=1$ or $i=l$), then $v_i$ is adjacent to exactly $d_{v_i}-1$ leaves in $L$.
\end{itemize}
\end{definition}

\begin{remark}
A path is a special case of a caterpillar tree, where the spine coincides with the entire tree. 
In this case, every internal edge lies on the spine and has curvature identically zero throughout the flow. 
Hence, under the Ricci flow, the metric naturally converges to a flat metric of zero constant curvature. 
\end{remark}

For any other caterpillar trees, we can prove the same result:

\begin{lemma}\label{lem:caterconverge}
Let $T$ be a caterpillar tree (not path) and let $w_e(t)$ evolve under the Ricci flow. Then
\begin{enumerate}
    \item Leaf curvatures converge: 
    \[
    \lim_{t\to\infty} \kappa_e(t) = 
    \begin{cases} 
    \frac{1}{d_{v_1}-1}, & e \text{ incident to } v_1\\
    \frac{1}{d_{v_l}-1}, & e \text{ incident to } v_l\\
    0, & \text{otherwise}.
    \end{cases}
    \]
    Internal edge curvatures converge to zero:
    \[
    \lim_{t\to\infty} \kappa_h(t) = 0.
    \]

    \item For any leaf $e$, 
    the normalized weight satisfies \[\lim_{t\to\infty}w_e(t)=0.\]

    \item 
    The product of unnormalized weight over all internal edges and leaves which are incident to $v_i$ for $i\neq 1, \ell$, is bounded. 

\end{enumerate}
\end{lemma}

\begin{proof}
 \textbf{Proof of result (1) }  The limit of curvature follows from Proposition~\ref{Pro:k_etendto0} and Corollary~\ref{pro:kf=0}. 

 \textbf{Proof of result (2) } 
This result is derived from Proposition~\ref{Pro:k_etendto0}. Within the proof, each leaf $e$ is connected to an internal edge $f$, where $\frac{w_e}{w_f}(t)\to 0$.

 \textbf{Proof of result (3) } 
Let $L$ be the set of leaf edges incident to $v_1$ or $v_{\ell}$. 
For finite $t$,  
\begin{align*}
  \sum\limits_{\text{leaf~}e\sim v_1}\kappa_e(t)  &=d_{v_1}-1-\frac{d_{v_1}-2}{D_{v_1}} \sum\limits_{\text{leaf~}e\sim v_1} \frac{1}{\tilde{w}_e}\\
  &>d_{v_1}-1 -(d_{v_1}-2)\\
  &>1, 
\end{align*}
similarly,  \[ \sum\limits_{\text{leaf~}e\sim v_{\ell}}\kappa_e(t) >1,\]

Then  
 \[\sum\limits_{e\in L}\kappa_e(t)\downarrow 2, \quad t\to \infty,\] 
 where $\downarrow 2$ means decrease to value $2$.
 
Recall from the proof of Proposition~\ref{pro:leaf=leaf} that for a pair of leaves $e, g$ incident to the same vertex $u$ with degree $d_u \ge 3$, the ratio of their weights satisfies
\[
\frac{w_e(t)}{w_g(t)} - 1 \le \Big(\frac{w_e(0)}{w_g(0)} - 1\Big) e^{\frac{2-d_u}{d_u} t}, 
\]
then $\frac{w_e(t)}{w_g(t)}$ converges exponentially to $1$ as $t \to \infty$.  

For pair of leaf $e$ and internal edge  $v_1v_2$ incident to $v_1$, we have  \[\frac{\partial}{\partial t} \log \frac{\tilde{w}_e}{\tilde{w}_{v_1v_2}}=(\kappa_{v_1v_2}-\kappa_e)\to  -\frac{1}{d_v-1},   \quad t\to \infty, \]
 then $\frac{\tilde{w}_e}{\tilde{w}_{v_1v_2}}$ exponentially decay to $0$. 

Noting that 
\begin{align*}
    \kappa_e& =1-\frac{d_{v_1}-2}{\sum\limits_{\text{leaf~}v_1x}\frac{\tilde{w}_e}{\tilde{w}_{v_1x}} + \frac{\tilde{w}_e}{\tilde{w}_{v_1v_2}}}, 
    \end{align*}
we have the exponential convergence for $\kappa_e(t)$ since the denominator exponentially converges. 
A similar result holds for $\kappa_e$ when $e\sim v_{\ell}$.

Hence, we  conclude 
\[
\sum_{e \in L} \kappa_e(t) -2\longrightarrow 0\quad \text{exponentially as } t \to \infty.
\]
Using
\[
-\sum_{uv \in E\setminus L} \kappa_{uv}(t) = \sum_{e \in L} \kappa_e(t)-2.
\]
Then 
\[
0< -\sum_{uv \in E\setminus L}\kappa_{uv}(t)  \leq  C_1 e^{-C_2 t}
\]
for some \(C_1, C_2>0\),
which implies that 
\[
\int_0^\infty -\sum_{uv \in E\setminus L}\kappa_{uv}(s) \, ds <  C_1/C_2.
\]

Using
\[
\prod_{uv \in E\setminus L} \tilde w_{uv}(t) = \Big(\prod_{uv \in E\setminus L} \tilde w_{uv}(0)\Big) \exp\Bigg( \int_0^t \big(\,-\sum_{uv \in E\setminus L} \kappa_{uv}(s)\,\big) ds \Bigg),
\]

we get 
\[
\prod_{uv \in E\setminus L} \tilde w_{uv}(t) 
 \leq  \Big(\prod_{uv \in E\setminus L} \tilde w_{uv}(0)\Big) e^{C_1/C_2},
\]
which is finite.
\end{proof}
Certain simulations involving caterpillar trees, as presented in the Appendix~\ref{appen}, show that the bounded product condition allows some unnormalized weights to become very large. 

\begin{remark}\label{re:initialaffect}
  The initial metric can influence the limiting weight value for a caterpillar tree. Consider a caterpillar tree where each spine vertex has a degree of $2$. Thus, leaf edges  only incident to the endpoints of the main path.

According to result (3) of Lemma~\ref{lem:caterconverge}, the product of unnormalized weights across all internal edges is finite, alongside the fact that the unnormalized weight on any internal edge is non-decreasing. then 
\[\lim\limits_{t\to \infty} \tilde{w}_h(t)<\infty ~~\text{for each internal edge~} h.\]

  Since the unnormalized weight at the leaves tends towards zero, we  have $$\lim\limits_{t\to \infty}\sum\limits_{h\in E} \tilde{w}_h(t) = \lim\limits_{t\to \infty}\sum\limits_{\text{internal }f\in E}  \tilde{w}_h(t)<\infty. $$

Now choose two initial unnormalized metrics, $\tilde{w}(0)$ and $ \tilde{w}^*_(0)$, such that $$\sum\limits_{h\in E}  \tilde{w}^*_h(0) >\lim\limits_{t\to \infty} \sum\limits_{h\in E}  \tilde{w}_h(t).$$

It follows that
   $$\lim\limits_{t\to \infty} \sum\limits_{h\in E}  \tilde{w}^*_h(t)>\lim\limits_{t\to \infty} \sum\limits_{h\in E}  \tilde{w}_h(t).$$
   
Now consider an internal edge $f$ whose endpoints are of degree $2$,  let the unnormalized weight satisfies $\tilde{w}_f(0)= \tilde{w}^*_f(0)$, then $ \tilde{w}_f(t)= \tilde{w}^*_f(t)$ for all $t$ as its derivative is always $0$.  Consequently, the normalized weight  on $f$ satisfies $$\lim\limits_{t\to \infty}  \tilde{w}^*_f(t)<\lim\limits_{t\to \infty}  \tilde{w}_f(t).$$
\end{remark}


\begin{lemma}\label{non-caterpillar}
In every non-caterpillar tree $T$, there are at least three internal vertices, denoted by $v$, each connected to precisely $d_v-1$ leaves.
\end{lemma}

\begin{proof}
By definition, a non-caterpillar tree contains at least one vertex $v_0$ with three or more internal edges.  
Since $T$ is finite, each internal edge of $v_0$ that does not continue indefinitely must eventually terminate at a leaf. 
Let $L_0$ denote the set of such leaves. 
Choose $l \in L_0$ so that, among all leaves in $L_0$ arising from the same branch at $v_0$, the distance $d(v_0,l)$ is maximal.  

If the edge incident to $l$ were adjacent to two or more internal edges, 
then one of these would not lie on the path from $v_0$ to $l$, 
and hence would produce another leaf $l'$ with $d(v_0,l') > d(v_0,l)$, a contradiction. 
Thus the edge incident to $l$ is adjacent to exactly one internal edge, 
so its neighbor $v$ (the internal vertex preceding $l$) is incident to $d_v-1$ leaves.  

Repeating this argument for at least three different branches from $v_0$ shows that there are at least three such internal vertex $v$ in $T$ that are adjacent to $d_v-1$ leaves. 
\end{proof}

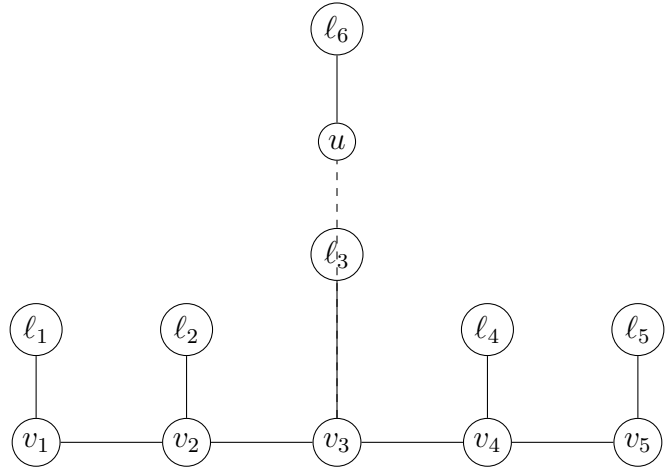
\begin{figure}[h!]
\centering
\begin{tikzpicture}[scale=1, every node/.style={circle,draw,inner sep=2pt}]
  \node (v1) at (0,0) {$v_1$};
  \node (v2) at (2,0) {$v_2$};
  \node (v3) at (4,0) {$v_3$}; 
  \node (v4) at (6,0) {$v_4$};
  \node (v5) at (8,0) {$v_5$};

  \node (l1) at (0,1.5) {$\ell_1$};
  \node (l2) at (2,1.5) {$\ell_2$};
  \node (l3) at (4,2.5) {$\ell_3$}; 
  \node (l4) at (6,1.5) {$\ell_4$};
  \node (l5) at (8,1.5) {$\ell_5$};

  \draw (v1) -- (v2) -- (v3) -- (v4) -- (v5);

  \draw (v1) -- (l1);
  \draw (v2) -- (l2);
  \draw (v3) -- (l3); 
  \draw (v4) -- (l4);
  \draw (v5) -- (l5);

  \node (u) at (4,4) {$u$};
  \draw[dashed] (v3) -- (u); 

  \node (l6) at (4,5.5) {$\ell_6$};
  \draw (u) -- (l6);
\end{tikzpicture}
\caption{A non-caterpillar tree: vertex $v_3$ has three internal edges (solid along spine and dashed extra branch), which forces at least one extra leaf $\ell_6$ not counted in the spine formula.}
\label{fig:non-caterpillar}
\end{figure}

\begin{lemma}
\label{lem:noncatinf}
Let $T$ be a non-caterpillar tree evolving under the Ricci flow \eqref{eq:unnor_continuous}
with unnormalized weights $\tilde w$. Then there exists at least one internal edge $f$ such that
\[
\lim_{t \to \infty} \tilde w_f(t) = \infty.
\]
Furthermore, the total of the unnormalized weights across all internal edges
\[
S(t) := \sum_{h \in H} \tilde w_h(t), 
\]
grows to infinity exponentially. 
\end{lemma}

\begin{proof}
There are two cases to consider.

Case 1. ~If there is a leaf $e=uv$ with $d_u\geq 3$ that satisfies Proposition~\ref{pro:kenegative},  we have $\kappa_e(t)<-\frac{d'_u}{d_u}$, then \[\tilde w_e(t)\geq \tilde w_e(t_0) e^{\frac{d'_u}{d_u}t}\] for some $t_0>0$. 
 Let $f\sim e$ be an internal edge adjacent to $e$, by Proposition~\ref{pro:leafinternal1},  \[\tilde w_f(t) \gg \tilde w_e(t),\] for large time $t$. The results ensued.

Case 2.  
Let $H$ denote the set of internal edges and $L$ the set of leaves. For any $e\in L$, either $\kappa_e(t)\to 0$ or $\kappa_e(t)\to \frac{1}{d_v-1}$ where $v$ denotes the internal vertex of $e$. 
By Lemma~\ref{non-caterpillar}, for a tree that is not a caterpillar, there are at least three internal vertices with leaves that lead to such positive curvature limits. 

Hence, for sufficiently large $t \ge t_0$,
\[
\sum_{e \in L} \kappa_e(t) \ge 2 + \delta
\]
for some $\delta > 0$. Using the total curvature formula $\sum_{e \in E(T)} \kappa_e(t) = 2$, we obtain
\[
\sum_{h \in H} \kappa_h(t) = 2 - \sum_{e \in L} \kappa_e(t) \le -\delta < 0, \quad t \ge t_0.
\]

Consider the product of internal edge weights
\[
P(t) := \prod_{h \in H} \tilde w_h(t).
\]
Differentiating along the Ricci flow:
\[
\frac{d}{dt} P(t) = \sum_{h \in H} \frac{d \tilde w_h}{dt} \prod_{h' \in H \setminus \{h\}} \tilde w_{h'} 
= - \sum_{h \in H} \kappa_h(t) \prod_{h \in H} \tilde w_h(t) 
= - \Big(\sum_{h \in H} \kappa_h(t)\Big) P(t).
\]

For $t \ge t_0$, $\sum_{h \in H} \kappa_h(t) \le -\delta$, so
\[
\frac{d}{dt} P(t) \ge \delta P(t) \quad \implies \quad P(t) \ge P(t_0) e^{\delta (t - t_0)}.
\]
Since $H$ is finite, the unbounded growth of $P(t)$ implies that at least one internal edge weight satisfies
\[
\limsup_{t \to \infty} \tilde w_f(t) = \infty.
\]
Since $\tilde w_f(t)$ is strictly increasing, then 
\[
\lim\limits_{t \to \infty} \tilde w_f(t) = \infty.
\]

Now consider the sum of the weights
\[
S(t) := \sum_{h \in H} \tilde w_h(t).
\]
By the AM–GM inequality,
\[
S(t) = \sum_{h \in H} \tilde w_h(t) \ge |H| \left( \prod_{h \in H} \tilde w_h(t) \right)^{1/|H|} = |H| \, P(t)^{1/|H|}.
\]
Using the exponential lower bound for $P(t)$, we obtain
\[
S(t) \ge |H| \, P(t_0)^{1/|H|} \, e^{\frac{\delta}{|H|} (t-t_0)}.
\]
Thus $S(t)$ grows at least exponentially.

\end{proof}

We then prove that for any non-caterpillar tree, if the curvature on an edge converges to zero, the edge weight cannot remain positive. Consequently, the tree cannot reach a metric with zero constant curvature.

\begin{lemma}\label{the:intenraledgefiniteweight}
Let T be any  non-caterpillar tree, let the metric $w$ of tree T evolve under Ricci flow with normalized weight $w$.
Let $f$ be an edge such that $\lim\limits_{t\to \infty} \kappa_f(t)=0, $ then  $\lim\limits_{t\to \infty} w_f(t)=0$. Moreover, $w_f(t)$ decays at least exponentially. 

\end{lemma}
\begin{proof}
For the edge $f$ with $\kappa_f(t) \to 0$, there exists $T \ge t_0$ such that $|\kappa_f(t)| \le \varepsilon$ for $t \ge T$, and thus
\[
\tilde w_f(t) = \tilde w_f(T) \exp\Big(-\int_T^t \kappa_f(s) \, ds \Big) \le \tilde w_f(T) e^{\varepsilon (t-T)}.
\]

We now consider the two cases corresponding to Lemma~\ref{lem:noncatinf}:

\textbf{Case 1.} There exists a leaf $e$ that satisfies Proposition~\ref{pro:kenegative}, then there exists $t_0>0$ such that $\kappa_e(t)\leq -\frac{d'_u}{d_u}$ for $t\geq T$.
It follows that \[\frac{\partial}{\partial t} \log \frac{\tilde{w}_f}{\tilde{w}_e} =(\kappa_e-\kappa_f)\leq -\frac{d'_u}{d_u}+\epsilon.\]
Then, we have the normalized weight
\[w_f(t)<\frac{1}{\frac{\tilde{w}_e}{\tilde{w}_f}}\leq \frac{\tilde{w}_f}{\tilde{w}_e}\leq \frac{\tilde{w}_f}{\tilde{w}_e}(t_0) e^{ (-\frac{d'_u}{d_u}+\epsilon)(t-T)}.\]

\textbf{Case 2.} Using the lower bound for $S(t)$ from Lemma~\ref{lem:noncatinf},  for $t \ge T$,
\[
w_f(t) < \frac{\tilde w_f(t)}{S(t)} \le \frac{\tilde w_f(T) e^{\varepsilon (t-T)}}{|H| P(t_0)^{1/|H|} e^{\frac{\delta}{|H|} (t-t_0)}} 
= C \, e^{-\frac{\delta}{|H|} (t-t_0) + \varepsilon (t-T)},
\]
for some constant $C>0$ depending on the initial weight.

For both cases, taking $\varepsilon$ arbitrarily small (since $\kappa_f(t) \to 0$), we obtain the exponential decay of the normalized weight $w_f(t)$. 

\end{proof}

\textbf{Proof of main Theorem \ref{thm:limexist0}}
\begin{proof}
  The result follows from Lemmas \ref{lem:caterconverge} and \ref{the:intenraledgefiniteweight} and Remark  \ref{re:initialaffect}.
\end{proof}
\begin{remark}
  We have established the curvature limit for the caterpillar tree and the normalized weight limit for all leaves of these trees. Yet, if any internal edges diverge to infinity, the existence of a normalized weight limit remains unresolved.
\end{remark}


\section{Exploring the Convergence of Curvature and Weights}

As stated in Conjecture \ref{conj:constant}, the key to the proof lies in demonstrating that the curvature on all edges converges.  
Once the conjecture is validated, we can derive a balanced condition satisfied by the limiting weights on the tree.

When the Ricci flow terminates, the tree may split into several distinct subtrees if any internal edges shrink to zero weight. These zero-weight internal edges, referred to as \textit{cutting edges}, partition the tree into subtrees where all edges retain positive weights.

Given the constant curvature across the positively weighted edges, the edge weights within each subtree are uniquely determined. Specifically, this corresponds to solving the system of equations 
\[
\kappa_{uv}(\infty) = \kappa, \quad \text{for every } uv \in E^+,
\]
where \(\kappa\) denotes the constant curvature. While this system is inherently non-linear, most of it can be linearized effectively by leveraging the following insights. 

Within each subtree, there exist pairs of \textit{terminal edges}, which are either leaf edges or internal edges adjacent to cutting edges. A \textit{maximal path} is defined as a sequence of edges connecting two terminal edges. The edge weights within each subtree adhere to specific balance conditions, as outlined below.

\begin{lemma}\label{lem:balancedweight}
Let \(T\) be a tree equipped with an initial metric \(w \in \mathbb{R}^{E(T)}_{>0}\). Suppose the metric \(w\) evolves under the Ricci flow on \(T\) and achieves a negative constant curvature at the limit. Then, for any positively weighted maximal path \(P = u_0 u_1 \ldots u_l\), the edge weights satisfy the following equations:
\begin{enumerate}
    \item If both ends of \(P\) are internal edges:
    \[
    \kappa w_1 - \kappa w_2 + \kappa w_3 - \cdots + (-1)^{l+1} \kappa w_l = 0.
    \]
    \item If the terminal edge \(w_l\) is a leaf edge:
    \[
    \kappa w_1 - \kappa w_2 + \kappa w_3 - \cdots + (-1)^{l+1} (\kappa - 1) w_l = 0.
    \]
    \item If both \(w_1\) and \(w_l\) are leaves:
    \[
    (\kappa - 1) w_1 - \kappa w_2 + \kappa w_3 - \cdots + (-1)^{l+1} (\kappa - 1) w_l = 0.
    \]
\end{enumerate}
\end{lemma}

\begin{proof}
We prove the first case, as the reasoning for the other two cases is analogous. Consider a maximal path \(P = u_0 u_1 \ldots u_l\), where \(l \geq 2\). At the limit, the system of equations governing the path is:
\begin{align*}
    0 + \frac{2 - d_{u_1}}{w_1 D_1} &= \kappa, \\
    \frac{2 - d_{u_1}}{w_2 D_1} + \frac{2 - d_{u_2}}{w_2 D_2} &= \kappa, \\
    \frac{2 - d_{u_2}}{w_3 D_2} + \frac{2 - d_{u_3}}{w_3 D_3} &= \kappa, \\
    &\vdots \\
    \frac{2 - d_{u_{l-2}}}{w_{l-1} D_{l-2}} + \frac{2 - d_{u_{l-1}}}{w_{l-1} D_{l-1}} &= \kappa, \\
    \frac{2 - d_{u_{l-1}}}{w_l D_{l-1}} + 0 &= \kappa.
\end{align*}

Multiplying each equation by the respective edge weight \(w_i\) and substituting recursively, we obtain:
\[
\kappa w_1 - \kappa w_2 + \kappa w_3 - \cdots + (-1)^{l+1} \kappa w_l = 0.
\]
The other cases follow similarly by adjusting the contributions of terminal leaf edges.
\end{proof}

\section{Appendix}\label{appen}
 
\begin{example}
We present experimental results of the unnormalized Ricci flow on: two caterpillar tree trees $T_1, T_3$ shown in  Figure \ref{fig:22} and Figure \ref{fig:t''}  and a non-caterpillar tree $T_2$ shown in Figure  \ref{fig:2}. 

The two trees $T_1$ and $T_2$ differ by only a single edge. On the internal edges, $T_1$ exhibits the curvature approaching $0$, with the unnormalized edge weights increasing to finite values. In contrast, in the $T_2$, the curvature tends toward $-1/3$, with the unnormalized weights growing unbounded.

Both $T_1$ and $T_3$ are caterpillar trees, differing by a single edge $x_3x_6$. Both converge to zero constant curvature under finite unnormalized weight. In $T_3$, the curvature of the leaf $x_3x_4$ initially rises from negative to positive before returning to zero, causing the unnormalized weight to increase and then decrease.

\begin{figure}[H]
\begin{minipage}[t]{0.3\textwidth}
 \centering
\begin{tikzpicture}[
  every node/.style={circle, draw, fill=white, inner sep=1pt, minimum size=0.2cm},
  leaf/.style={circle, draw, fill=black, inner sep=1pt, minimum size=0.2cm}
]

\node (x5) at (0,0) {$x_5$};
\node (x3) at (1,0) {$x_3$};
\node (x2) at (2,0) {$x_2$};
\node (x1) at (3,1) {$x_1$};
\node (u1) at (3,-1) {$u_1$};

\node (x4) at (1,2) {$x_4$};
\node (u3) at (-1,-1) {$u_3$};
\node (u2) at (-1,1) {$u_2$};

\draw (x3) -- (x2);
\draw (x3) -- (x5);
\draw (x1) -- (x2);
\draw (u1) -- (x2);
\draw (x3) -- (x4);
\draw (u2) -- (x5);
\draw (u3) -- (x5);

\end{tikzpicture}
\end{minipage}
\hfill
  \begin{minipage}[t]{0.6\textwidth}
 \centering
\includegraphics[width=0.8\linewidth]{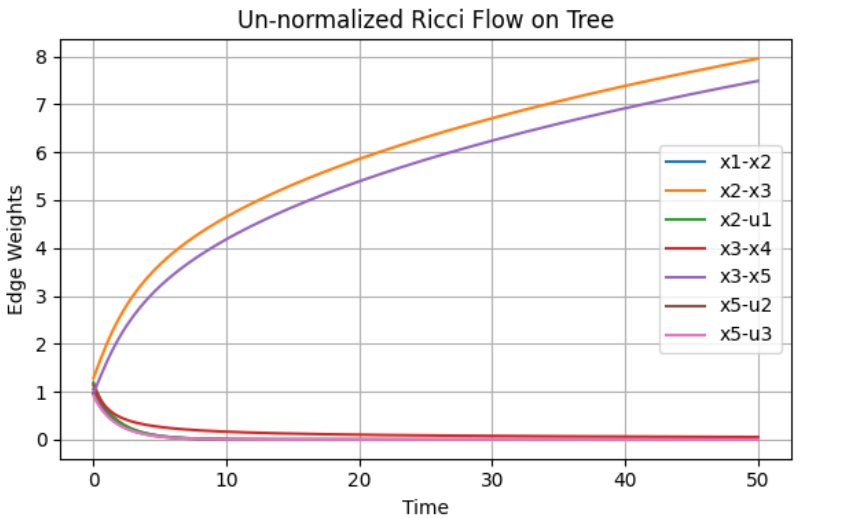}
\end{minipage}
 \caption{An evolution of the unnormalized Ricci flow on the left tree $T_1$
  with constant curvature $0$}\label{fig:22} 
\end{figure}
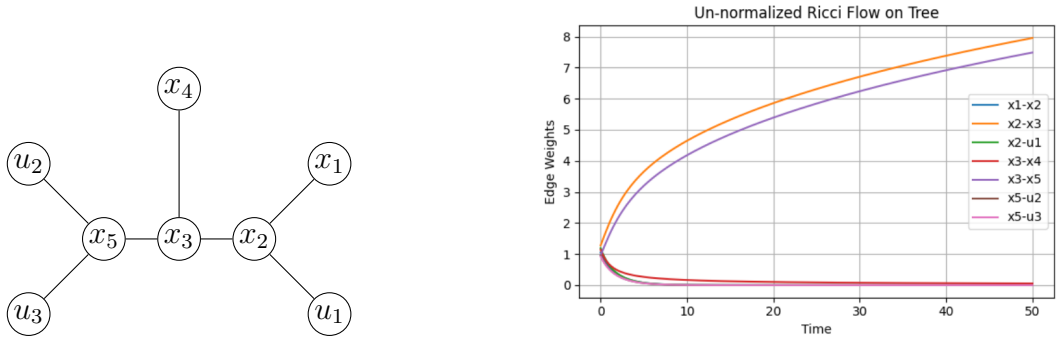

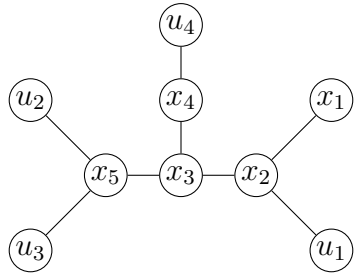
\begin{figure}[!h]
\centering
\begin{tikzpicture}[
  every node/.style={circle, draw, fill=white, inner sep=1pt, minimum size=0.4cm}
]

\node (x5) at (0,0) {$x_5$};
\node (x3) at (1,0) {$x_3$};
\node (x2) at (2,0) {$x_2$};
\node (x1) at (3,1) {$x_1$};
\node (u1) at (3,-1) {$u_1$};

\node (x4) at (1,1) {$x_4$};
\node (u3) at (-1,-1) {$u_3$};
\node (u2) at (-1,1) {$u_2$};
\node (u4) at (1,2) {$u_4$};

\draw (x3) -- (x2);
\draw (x3) -- (x5);
\draw (x1) -- (x2);
\draw (u1) -- (x2);
\draw (x3) -- (x4);
\draw (u2) -- (x5);
\draw (u3) -- (x5);
\draw (u4) -- (x4);

\end{tikzpicture}
\caption{Tree $T_2$}\label{fig:2} 
\end{figure}

\begin{figure}[!h]
  \begin{minipage}[t]{0.45\textwidth}
 \centering
 \includegraphics[width=0.9\linewidth]{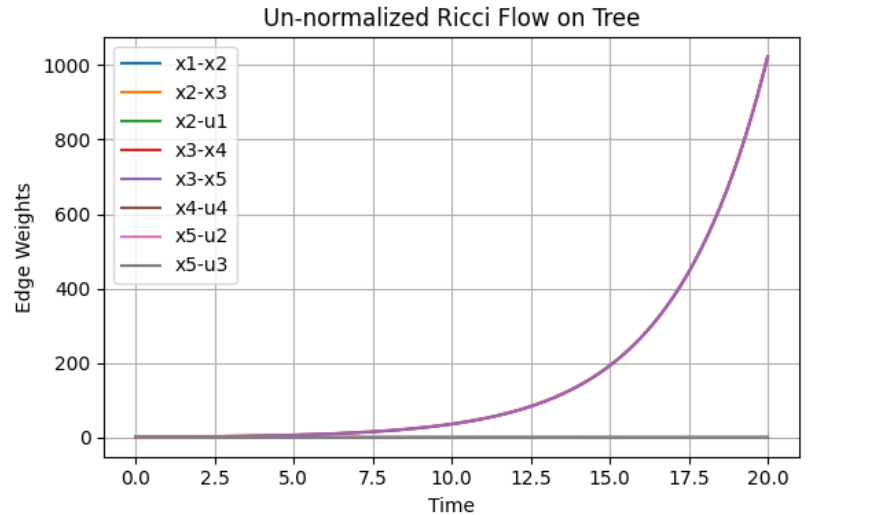}
\end{minipage}
  \hfill
  \begin{minipage}[t]{0.45\textwidth}
 \centering
\includegraphics[width=0.9\linewidth]{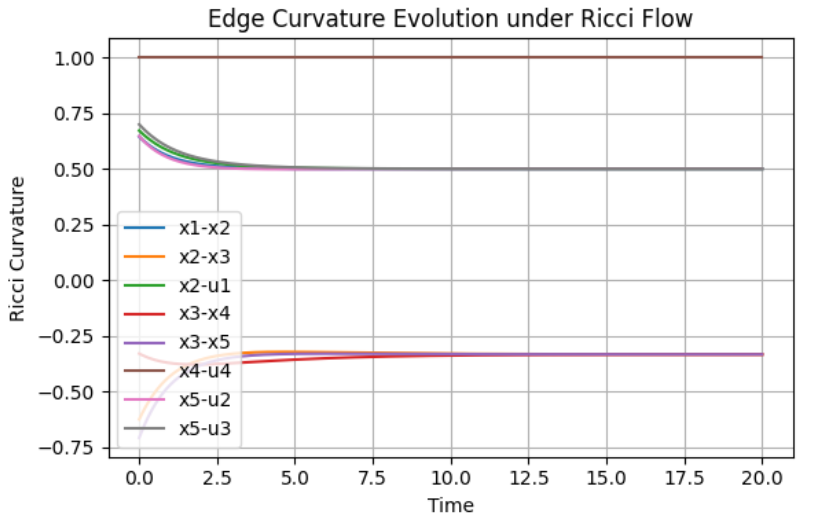}
      \end{minipage}
          \caption{An evolution of the unnormalized Ricci flow on the tree $T_2$
  with constant curvature $-\frac{1}{3}$.}\label{fig:3}
\end{figure}

\begin{figure}[H]
\centering
\begin{tikzpicture}[
  every node/.style={circle, draw, fill=white, inner sep=1pt, minimum size=0.2cm},
  leaf/.style={circle, draw, fill=black, inner sep=1pt, minimum size=0.2cm}
]

\node (x5) at (0,0) {$x_5$};
\node (x3) at (1,0) {$x_3$};
\node (x2) at (2,0) {$x_2$};
\node (x1) at (3,1) {$x_1$};
\node (u1) at (3,-1) {$u_1$};

\node (x4) at (0.5,2) {$x_4$};
\node (x6) at (1.5,2) {$x_6$};
\node (u3) at (-1,-1) {$u_3$};
\node (u2) at (-1,1) {$u_2$};

\draw (x3) -- (x2);
\draw (x3) -- (x5);
\draw (x1) -- (x2);
\draw (u1) -- (x2);
\draw (x3) -- (x4);
\draw (x3) -- (x6);
\draw (u2) -- (x5);
\draw (u3) -- (x5);

\end{tikzpicture}
\caption{Tree $T_3$}\label{fig:t''} 
\end{figure}
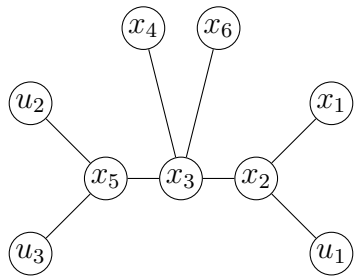

\bigskip

\begin{figure}[H]
    \centering
\includegraphics[width=0.45\linewidth]{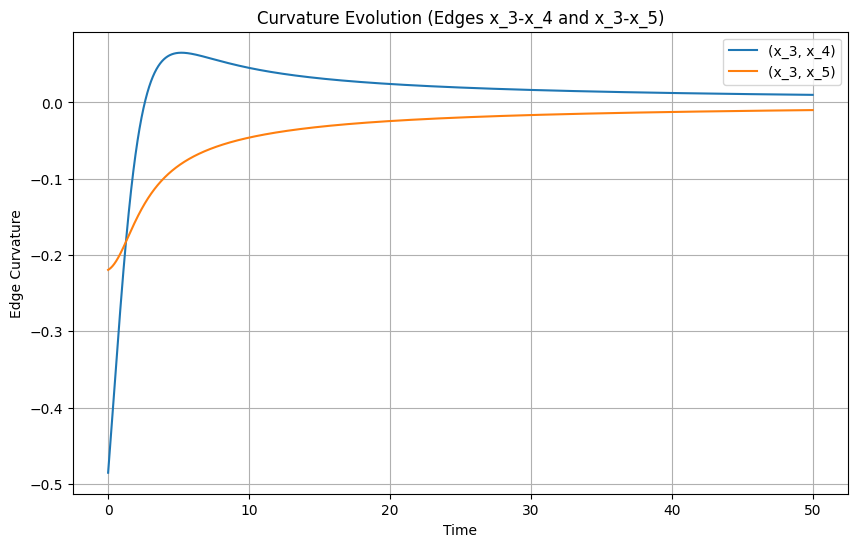}
  \includegraphics[width=0.45\linewidth]{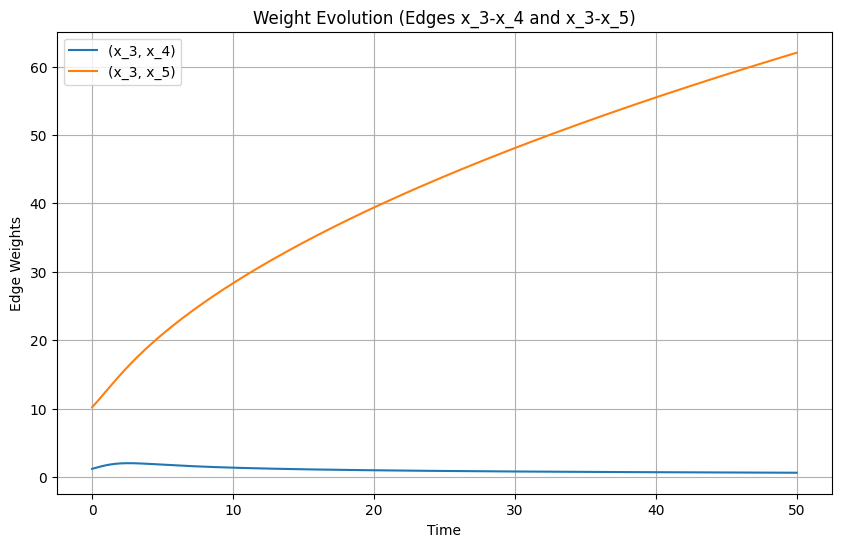}
          \caption{An evolution of the unnormalized Ricci flow on the tree $T_3$  with constant curvature $0$. }\label{fig:tree0curvature4} 
\end{figure}
For above caterpillar tree  $T_3$, numerical simulations indicate that the weight of the edge $x_3x_5$ increases monotonically over time, showing no sign of convergence to finite values. On a long time scale, the curve continues to rise, supporting the conjecture that its limit is unbounded. 

\end{example}

{\bf Acknowledgements:} 

We are grateful to Professor S.-T. Yau for proposing the problem of the convergence of Ricci flow on graphs and for numerous insightful discussions on the subject.

S. Bai is supported by NSFC, no.12301434. B. Hua is supported by NSFC, no.12371056.
Y. Lin is supported  by NSFC, no. 12471088. S. Liu is supported by NSFC, no.12001536, 12371102.

\bibliographystyle{plain}
\bibliography{citations}

\end{document}